\documentclass[a4paper,11pt]{amsart}
\usepackage{amssymb}
\usepackage{mathrsfs}
\usepackage{amsmath,amssymb,amsthm,latexsym,amscd,mathrsfs}
\usepackage{indentfirst}
\usepackage{stmaryrd}
\usepackage{graphicx}
\usepackage{extarrows}
\usepackage{amssymb}
\usepackage{amsmath,amssymb,amsthm,latexsym,amscd,mathrsfs}
\usepackage{indentfirst}
\usepackage{multirow}
\usepackage{latexsym}
\usepackage{amsfonts}
\usepackage{color}
\usepackage{pictexwd,dcpic}
\usepackage{graphicx}
\usepackage{psfrag}
\usepackage{hyperref}
\usepackage{comment}
\usepackage{cases}

 \newcommand{\ROM}[1]{\mathrm{\uppercase\expandafter{\romannumeral#1}}}
  \theoremstyle{definition}
   \numberwithin{equation}{section} \theoremstyle{plain}

 \newtheorem{thm}{Theorem}[section]
 \newtheorem{lem}{Lemma}[section]
 \newtheorem{defn}{Definition}[section]
 \newtheorem{cor}{Corollary}[section]
  
 \newtheorem{rem}{Remark}[section]
 \newtheorem{prop}{Proposition}[section]

 \newtheorem{assert}{Assertion}[subsection]
  \newtheorem{observ}{Observation}[subsection]

\newtheorem{ack}{Acknowledgements}   
  \makeatletter

  \numberwithin{equation}{section}
  \allowdisplaybreaks
\makeatother
\title[Normal scalar curvature inequality ]{\textbf{Normal scalar curvature inequality on the focal submanifolds of isoparametric hypersurfaces}}
\author[J. Q. Ge]{Jianquan Ge}\address{School of Mathematical Sciences, Laboratory of Mathematics and Complex Systems, Beijing Normal
University, Beijing 100875, P. R. China}\email{jqge@bnu.edu.cn}
\author[Z. Z. Tang]{Zizhou Tang}\address{Chern Institute of Mathematics $\&$ LPMC, Nankai University, Tianjin 300071, P. R. China}
\email{zztang@nankai.edu.cn}
\thanks {$^{\dag}$ the corresponding author}
\author[W. J. Yan]{Wenjiao Yan$^{\dag}$}\address{School of Mathematical Sciences, Laboratory of Mathematics and Complex Systems, Beijing Normal University, Beijing 100875, P. R. China}
\email{wjyan@bnu.edu.cn}
\thanks {The project is partially supported by the NSFC (No.11331002, 11301027, 11522103, 11722101)}


\begin{document}

\maketitle
\begin{center}
Dedicated to Professor Chiakuei Peng on the Occasion of His 75th Birthday
\end{center}
\begin{abstract}
An isoparametric hypersurface in unit spheres has two focal submanifolds.
Condition A plays a crucial role in the classification theory of isoparametric hypersurfaces in \cite{CCJ07}, \cite{Chi16} and \cite{Miy13}.
This paper determines $C_A$, the set of points with Condition A in focal submanifolds. It turns out that the points in
$C_A$ reach an upper bound of the normal scalar curvature $\rho^{\bot}$ (sharper than that in DDVV inequality \cite{GT08}, \cite{Lu11}).
We also determine the sets $C_P$ (points with parallel second fundamental form) and $C_E$ (points with Einstein condition), which achieve
two lower bounds of $\rho^{\bot}$.
\end{abstract}

\section{\textbf{Introduction}}
A smooth function $f: N\rightarrow \mathbb{R}$ defined on a Riemannian manifold $N$ is called \emph{isoparametric} (cf.\cite{Wan87},\cite{GT13},\cite{QT15}), if there exist a smooth function $b:\mathbb{R}\rightarrow\mathbb{R}$ and a continuous function $a:\mathbb{R}\rightarrow\mathbb{R}$ such that
\begin{equation}\label{iso1}
\left\{ \begin{array}{ll}
|\nabla f|^2=b(f),\\
~~\,\,\Delta f~~=a(f),
\end{array}\right.
\end{equation}
where $\nabla f$, $\Delta f$ are the gradient and Laplacian of $f$, respectively.
The regular level sets of $f$ are called isoparametric hypersurfaces in $N$, while
the singular ones (if exist) are called the focal varieties, which are smooth submanifolds of $N$.
For $N=S^{n+1}(1)$, a well-known result of E. Cartan states that an isoparametric hypersurface in $S^{n+1}(1)$
 is nothing but a hypersurface with constant principal curvatures.
Without loss of generality, suppose now that $\mathrm{Image}(f)=[-1,1]$.
The focal varieties of $f$, $M_{+}=:f^{-1}(+1)$ and $M_{-}=:f^{-1}(-1)$, are connected, minimal submanifolds of $S^{n+1}(1)$ (cf. \cite{CR}).

Quite recently, the classification of isoparametric hypersurfaces in $S^{n+1}(1)$ was accomplished by Miyaoka and Chi. This settles S.T.Yau's
34th problem in \cite{Yau82}. To be more precise,
let $g$ be the number of distinct principal curvatures with multiplicity $m_i$ ($i=1, \cdots, g$). According to \cite{Mun},
$g$ can be only $1, 2, 3, 4$ or $6$, and $m_i=m_{i+2}$ (subscripts mod $g$). When $g=1, 2, 3$ and $6$, the isoparametric
hypersurfaces are homogeneous (cf. \cite{DN85}, \cite{Miy13}, \cite{Miy16} );
when $g=4$, the isoparametric hypersurfaces are either of OT-FKM type (defined later), or homogeneous with $(m_1, m_2) =(2, 2), (4, 5)$
(cf. \cite{CCJ07}, \cite{Chi11}, \cite{Chi13}, \cite{Chi16}).
In these papers towards the classification problem, the Condition A always plays a very important role in the approach (cf. \cite{Chi12}).

The geometric meaning of Condition A
is that at some point in a focal submanifold, the kernels of all second fundamental tensors (shape operators) coincide (cf. \cite{OT75}). It is actually satisfied everywhere on each focal submanifold when $g=1,2,3$ and $6$. The proofs are straightforward for the former three cases but far from trivial for $g=6$ (cf. \cite{Miy93}, \cite{Miy13}). In fact, Condition A is usually
used to detect the homogeneity and inhomogeneity. By using delicate isoparametric triple system,
\cite{DN83} succeeded in classifying all the isoparametric hypersurfaces with $g=4$, on which Condition A is fulfilled (on one focal submanifold). Translating their language of isoparametric triple systems
into a plainer form, one has

\noindent \textbf{Theorem (\cite{DN83})}\,\,
\textit{On the focal submanifolds of isoparametric hypersurfaces in $S^{n+1}(1)$ with $g=4$, there are points satisfying Condition A on and only on the following:}
\textit{\begin{itemize}
  \item [(1).] The focal submanifolds $M_+$ of OT-FKM type with $(m_1, m_2)=(1,k), (3, 4k)$ and $(7, 8k)$ ($k\geq 1$);
  \item [(2).] The focal submanifolds $M_-$ of OT-FKM type with $(m_1, m_2)=(4,3)$ and $(8,7)$;
  \item [(3).] The focal submanifold diffeomorphic to $\mathbb{C}P^3$ with $(m_1, m_2)=(2,2)$.\footnote{In the $(2, 2)$ case, one focal submanifold is diffeomorphic to $\mathbb{C}P^3$ and the other diffeomorphic to $\widetilde{G}_2(\mathbb{R}^5)$ (cf. \cite{QTY13}). In fact, \cite{DN83} and \cite{OT76} did not point out which one of, or whether both of the focal submanifolds satisfy Condition A. This will be determined in this paper by using the normal scalar curvature inequalities (\ref{1st}, \ref{2nd}).}
\end{itemize}}
Here comes a natural question: which points on these focal submanifolds satisfy Condition A ?  Given an isoparametric hypersurface in $S^{n+1}(1)$ with $g=4$, we define
the set
$$C_A=:\{ \mathrm{the~ points~ satisfying~ Condition~ A}\}.$$
As one of the main purposes of this paper, we aim at determining $C_A$ for all the focal submanifolds mentioned in the theorem above.
Here we should point out that \cite{DN83'} gave a partial answer to this question in $M_-$ of OT-FKM type with multiplicities $(4, 3)$ and $(8, 7)$. In fact, they found subsets of $C_A$ in $M_-$ under an additional assumption that the point is an eigenvector of the fixed Clifford matrix $P_0$.

Next, let us explain Condition A in a more accurate way. On a focal submanifold $M_+$ with $g=4$ and codimension $m_1+1$ in $S^{n+1}(1)$,
it is well known that the shape operator associated with any unit normal vector has eigenvalues $+1$, $-1$ and $0$ with multiplicities $m_2$, $m_2$ and $m_1$, respectively
( as for $M_-$, replace the corresponding multiplicities $m_1$, $m_2$ with $m_2$, $m_1$). Let $\{n_{\alpha}, ~\alpha=0,1,\cdots, m_1\}$ be an orthonormal basis of the normal space $T_x^{\bot}M_+$. Following \cite{OT75}, we denote by $V_+$, $V_-$ and $V_0$ the eigenspaces of the shape operator in the normal direction $n_0$ associated with the eigenvalues $1$, $-1$ and $0$, respectively.
With this understood, the shape operators $S_{\alpha}=:S_{n_{\alpha}}$ ($0\leq \alpha\leq m_1$), upon certain fixing orthonormal basis $\{e_1,\cdots, e_{m_1+2m_2}\}$ of $T_x M_+$, can be written as:
\begin{equation}\label{sff}
 S_0=
\left(
  \begin{array}{ccc}
    I_{m_2} &  &  \\
     & -I_{m_2} &  \\
     &  & 0 \\
  \end{array}
\right),\quad
S_{\alpha}=
\left(
  \begin{array}{ccc}
    0 & A_{\alpha} & B_{\alpha} \\
    A^t_{\alpha} & 0 & C_{\alpha} \\
    B^t_{\alpha} & C^t_{\alpha} & 0 \\
  \end{array}
\right), 1\leq \alpha\leq m_1,
\end{equation}
where $A_{\alpha}: V_-\rightarrow V_+$, $B_{\alpha}: V_0\rightarrow V_+$ and $C_{\alpha}: V_0\rightarrow V_-$. Then a point
$x\in M_+$ is of Condition A if and only if $B_{\alpha}=C_{\alpha}=0$ for all $1\leq \alpha\leq m_1$.

Moreover, we find an interesting phenomenon that Condition A can be alternatively characterized by the normal scalar curvature (function) inequality on the focal submanifolds.
In fact, the normal scalar curvature (function) achieves an upper bound at the points satisfying Condition A, and this upper bound is sharper than that in the DDVV (normal scalar curvature) inequality.

Recall that the normal scalar curvature of $M_+^{m_1+2m_2}$ in $S^{2m_1+2m_2+1}(1)$ can be defined as (slightly different from that in \cite{DDVV99})
\begin{equation}\label{normalscalar}
\rho^{\bot}=:\sum_{\alpha,\beta=0}^{m_1}\|[S_{\alpha}, S_{\beta}]\|^2=\sum_{i,j=1}^{m_1+2m_2}\sum_{\alpha,\beta=0}^{m_1}\langle R^{\bot}(e_i,e_j)n_\alpha, n_\beta\rangle^2=\|R^{\bot}\|^2,
\end{equation}
where $R^{\bot}$ is the normal curvature tensor. We remark that the second equality above follows from the Gauss and Ricci equations.
Notice that the squared norm of the second fundamental form $S$
is constant on the focal submanifolds, more precisely, $\|S\|^2=2m_2(m_1+1)$ on $M_+$ ($g=4$), (resp. $2m_1(m_2+1)$ on $M_-$).
By the minimality of focal submanifolds in $S^{n+1}(1)$, the DDVV inequality on $M_+$ (resp. $M_-$ with $(m_1,m_2)$ interchanged) ($g=4$ for example)
can be transformed into (cf. \cite{DDVV99}, \cite{GT08}, \cite{Lu11})
\begin{eqnarray}\label{DDVV}
\rho^{\bot} &\leq& 4m_2^2(m_1+1)^2.
\end{eqnarray}
Similarly, it is not difficult to obtain the corresponding forms of the DDVV inequality on focal submanifolds with $g\neq 4$.
Generally speaking, the submanifolds in space forms whose normal scalar curvature achieves the equality of the DDVV inequality everywhere (the so-called Wintgen ideal submanifolds) are still unclassified, although many partial results and studies are available in the literature (cf. \cite{CL08}, \cite{DT09}, \cite{XLMW14}, etc.).
When restricted to the focal submanifolds of isoparametric hypersurfaces in unit spheres, the DDVV inequality hardly achieves the equality even at a point.
Fortunately, by the pointwise equality condition of the DDVV inequality and the special properties of the shape operators of focal submanifolds, we find two examples of the Wintgen ideal submanifolds:
\begin{prop}\label{DDVVequality}
On the focal submanifolds of isoparametric hypersurfaces in $\mathbb{S}^{n+1}(1)$, the DDVV inequality achieves equality (everywhere) on and only on:
\begin{itemize}
\item[(1).] $M_+\cong M_-\cong \mathbb{R}P^2\subset \mathbb{S}^4$ in the case of $g=3$, $(m_1,m_2)=(1,1)$;
\item[(2).] $M_+\cong SO(3)\subset \mathbb{S}^5$ in the OT-FKM type of $g=4$, $(m_1,m_2)=(1,1)$.
\end{itemize}
 \end{prop}
We leave the proof to readers, although the proof of ``\emph{only on}" is not trivial (one can either calculate directly by using the properties of the OT-FKM type, or simply compare with the classification of $C_A$ and $C_P$ in Theorem \ref{thm1} below). It should be remarked that \cite{XLMW14} studied the M\"{o}bius geometry of three dimensional Wintgen ideal submanifolds in $S^5(1)$, and the second example in the proposition above is exactly theirs.

Next, using the formulas (\ref{sff}, \ref{normalscalar}), we shall derive the following new normal scalar curvature inequalities on the focal submanifold $M_+$ (resp. $M_-$ with $(m_1,m_2)$ interchanged):
\begin{eqnarray}\label{1st}
2m_1m_2(m_1+1)~\leq~ \rho^{\bot}~\leq~ 8m_1m_2(m_1+1).
\end{eqnarray}
Notice that the upper bound in (\ref{1st}) is shaper than that in the DDVV inequality (\ref{DDVV}).
We shall prove later that this upper bound can be reached exactly at
points of Condition A, whose set was denoted by $C_A$ as before.
On the other hand, the first equality in (\ref{1st}) holds at points with parallel second fundamental form. We define
$$C_P=:\{ \mathrm{the~ points~ with~ parallel~ second~ fundamental~ form}\}.$$
In Section 2, we will express the normal scalar curvature $\rho^{\bot}$ in an alternative way as
\begin{equation}\label{another rho}
\rho^{\bot}=6\|\sum_{\alpha=0}^{m_1}S_{\alpha}^2\|^2-4m_2(m_1+1)(m_1+3).
\end{equation}
Then using Schwartz inequality, we obtain a sharper lower bound for $\rho^{\bot}$ on $M_+$ (resp. $M_-$ with $(m_1,m_2)$ interchanged):
\begin{equation}\label{2nd}
\rho^{\bot}\geq 2m_1m_2(m_1+1)+\frac{6m_1m_2(m_1+1)}{2m_2+m_1}(2m_2-m_1-2),
\end{equation}
where the equality holds at points of Einstein condition, i.e., the Ricci curvature is constant at that point.
We denote
$$C_E=:\{ \mathrm{the~ points~ with~ constant~ Ricci~ curvature~ in~ all~ tangent~ directions~}\}.$$

Summarizing all the estimates above, and combining with the results in \cite{TY15}, \cite{TY17}, \cite{QTY13} and \cite{LZ16},
we are able to prove the following theorem, determining all the sets $C_A$, $C_P$ and $C_E$.

\begin{thm}\label{thm1}
Let $M_+$ (resp. $M_-$) be a focal submanifold of an isoparametric hypersurface with $g=4$ and
codimension $m_1+1$ (resp. $m_2+1$) in $S^{n+1}(1)$. Then we have the normal scalar curvature inequalities
(\ref{1st}) and (\ref{2nd}) on $M_+$ (resp. $M_-$). The following Table \ref{table-M+} and \ref{table-M-} give all the sets
$C_A$, $C_P$ and $C_E$ in $M_+$ and $M_-$. The first equality in (\ref{1st}) holds on and only on $C_P$, the second one in (\ref{1st})
holds on and only on $C_A$, and the equality in (\ref{2nd}) holds on and only on $C_E$.
Here, except for $(2,2)$, all multiplicities $(m_1,m_2)$ stand for $(m,l-m-1)$ in OT-FKM type; $D$ and $I$ are short for definite and indefinite.
$\Omega_{k}=:(S^k\times S^7)/\mathbb{Z}_2,$ for convenience.
\begin{center}{\scriptsize
\begin{table}[!htb]
\caption{On the focal submanifold $M_+$}\label{table-M+}
\begin{tabular}{|c|c|c|c|c|c|c|c|c|c|}
\hline
     & $(1,k)$   & $(2,1)$   & $(6,1)$   & $(3,4k)$                  &  $(4,3)$D &  $(7,8k)$                  & $(2,2)$& others   \\
\hline
$C_A$& $M_+$   & $\emptyset$      &   $\emptyset$      & $S^{3+4k}\sqcup S^{3+4k}$ &   $\emptyset$                   & $\Omega_{k}\sqcup\Omega_{k}$ &   $\mathbb{C}P^3$ & $\emptyset$\\
\hline
$C_P$&  $\emptyset$     &  $M_+$ &$M_+$ &   $\emptyset$     &    $M_+$            & $\emptyset$    &  $\emptyset$& $\emptyset$ \\
\hline
$C_E$&  $\emptyset$     &   $\emptyset$   &   $\emptyset$      &  $\emptyset$      &    $M_+$        &  $\emptyset$     & $\emptyset$ & $\emptyset$\\
\hline
\end{tabular}
\end{table}}
\end{center}

\begin{center}{\scriptsize
\begin{table}[!htb]
\caption{On the focal submanifold $M_-$}\label{table-M-}
\begin{tabular}{|c|c|c|c|c|c|c|c|c|c|c|}
\hline
     & $(1,k)$ & $(2, 1)$ & $(6,1)$&  $(4,3)$D  & $(4,3)$I       & $(8,7)$D                     &  $(8,7)$I    & $(2,2)$& others   \\
\hline
$C_A$&  $\emptyset$ & $M_-$ & $M_-$ &  $M_-$         & $S^{7}\sqcup S^{7}$ &  $\tiny{(S^1\times S^{15})/\mathbb{Z}_2}$    & $S^{15}\sqcup S^{15}$  & $\emptyset$ &$\emptyset$\\
\hline
$C_P$&  $M_-$  & $\emptyset$        & $\emptyset$ & $\emptyset$    &  $\emptyset$    &$\emptyset$  &$\emptyset$ & $\widetilde{G}_2(\mathbb{R}^5)$& $\emptyset$\\
\hline
$C_E$&  $\emptyset$   & $\emptyset$       & $\emptyset$ &    $\emptyset$         &$\emptyset$ & $\emptyset$    &      $\emptyset$   & $\widetilde{G}_2(\mathbb{R}^5)$& $\emptyset$\\
\hline
\end{tabular}
\end{table}}
\end{center}
\end{thm}
\begin{rem}\label{rem1}
According to Theorem 6.5 in \cite{FKM81}, the isoparametric families of OT-FKM type with multiplicities $(2, 1)$, $(6, 1)$ are
congruent to those with multiplicities $(1, 2)$, $(1, 6)$; The indefinite $(4, 3)$ family is congruent to the $(3, 4)$ family.
Therefore, there are some coincidences of the sets $C_A$, $C_P$ and $C_E$ in Tables \ref{table-M+} and \ref{table-M-}.  
\end{rem}

\begin{rem}
For $g=3$, as it is well known, the focal submanifolds are Veronese embeddings of projective planes $\mathbb{F}P^2$, $\mathbb{F}=\mathbb{R},\mathbb{C},\mathbb{H},\mathbb{O}$, thus they are Einstein. Moreover, they are examples of symmetric $R$-spaces (\cite{KT68}), and thus have parallel second fundamental form. As a result, $C_P=C_E=M_{\pm}$. For $g=6$, all focal submanifolds have non-vanishing covariant derivative of second fundamental form (\cite{LZ16}) and are not Einstein (\cite{Xie}), thus $C_P=C_E=\emptyset$ by the homogeneity (\cite{DN85, Miy13}).
So far we have determined all the three sets where the normal scalar curvature inequalities (\ref{1st}, \ref{2nd}) reach equality for any isoparametric hypersurface in unit spheres.
Conversely, if one could prove directly that the normal scalar curvature function achieves its maximum at some point in the inequality (\ref{1st}), then the isoparametric hypersurfaces with $g=4$ are of OT-FKM type or with multiplicities $(2,2)$. In particular, this approach would provide a new geometric proof for the the last classified case $(7,8)$ done by \cite{Chi16}, where deep algebraic geometry method was developed (see for a detailed discussion in \cite{Chi15}).
\end{rem}
Next we introduce the isoparametric hypersurfaces of OT-FKM type which exhaust almost all of the isoparametric hypersurfaces with $g=4$.
For a symmetric Clifford system $\{P_0,\cdots,P_m\}$ on $\mathbb{R}^{2l}$, \emph{i.e.}
$P_{\alpha}$'s are symmetric matrices satisfying $P_{\alpha}P_{\beta}+P_{\beta}P_{\alpha}=2\delta_{\alpha\beta}I_{2l}$,
a homogeneous polynomial $F$ of degree $4$ on $\mathbb{R}^{2l}$ is defined as:
\begin{eqnarray}\label{FKM isop. poly.}
&&F(x) = |x|^4 - 2\displaystyle\sum_{\alpha = 0}^{m}{\langle
P_{\alpha}x,x\rangle^2}.
\end{eqnarray}
It is easy to verify that $f=F|_{S^{2l-1}(1)}$ is an isoparametric function
on $S^{2l-1}(1)$, which is said to be of OT-FKM type following the names of Ozeki, Takeuchi, Ferus, Karcher and M\"{u}nzner (\cite{OT75,OT76,FKM81}).
The focal submanifolds are $M_+=f^{-1}(1)$, $M_-=f^{-1}(-1)$, and the multiplicity pair is $(m_1, m_2)=(m, l-m-1)$, provided
$m>0$ and $l-m-1>0$, where $l = k\delta(m)$ $(k=1,2,3,\cdots)$, $\delta(m)$ is the dimension
of an irreducible module of the Clifford algebra $\mathcal{C}_{m-1}$. 
When $m\equiv 0(mod ~4)$, according to \cite{FKM81}, there are two kinds of OT-FKM type isoparametric
polynomials which are distinguished by $\mathrm{Trace}(P_0P_1\cdots P_m)$, namely, the family with $P_0P_1\cdots P_m=\pm Id$, where without loss of generality
we take the $+$ sign, called the \emph{definite} family, and the others with $P_0P_1\cdots P_m\neq\pm Id$,
called the \emph{indefinite} family. There are exactly $[\frac{k}{2}]$ non-congruent indefinite families.

The isoparametric hypersurfaces of OT-FKM type with multiplicities $(8, 7)$ have a wealth of geometric contents,
being also the last case in the process of classification (cf. \cite{Chi16}).
During our investigation for the set $C_A$ in Theorem \ref{thm1}, we find very intriguing relations between the focal submanifolds in the definite and
indefinite families (now only one indefinite family).
Denote by $I$ the indefinite family, that is, $P=:P_0P_1\cdots P_8\neq \pm Id$. It is easily seen that $\{\tilde{P_0}=:P_0P, \cdots, \tilde{P_8}=:P_8P\}$ constitutes
a new symmetric Clifford system with $\tilde{P_0}\tilde{P_1}\cdots \tilde{P_8}=Id$, which corresponds to the definite family, denoted by $D$.
We add superscripts $I$ and $D$ to distinguish the corresponding focal submanifolds and the sets $C_A$. As a byproduct, we obtain

\begin{prop}\label{prop}
Let $M_{+}$, $M_-$ be the focal submanifolds of isoparametric hypersurfaces of OT-FKM type with multiplicities $(8, 7)$ in $S^{n+1}(1)$.
Then we can define isoparametric functions $h$ on $M_-^I$ and $\tilde{h}$ on $M_-^D$, whose isoparametric hypersurfaces $h^{-1}(0)$ and $\tilde{h}^{-1}(0)$:
$$M_+^D \hookrightarrow M_-^I,\quad M_+^I \hookrightarrow M_-^D$$
are totally isoparametric, austere hypersurfaces. Moreover, the focal variety of $h$ on $M_-^I$ is exactly the
set $C_A^I\subset M_-^I$.
\end{prop}

Let us now explain the definition of austerity that appeared
above. Recall that in a Riemannian manifold, isoparametric hypersurfaces defined by (\ref{iso1}) are a family of parallel hypersurfaces with
constant mean curvatures. Extending this concept, \cite{GTY15} defined the object \emph{k-isoparametric hypersurfaces}, whose nearby parallel hypersurfaces
have constant $i$-th mean curvatures for $i=1,\cdots, k$. In particular, \emph{totally isoparametric hypersurfaces}
are those whose nearby parallel hypersurfaces have constant principal curvatures. We would like to point out that in real space forms, an isoparametric hypersurface
is totally isoparametric, while in other spaces it is usually not the case (cf. \cite{GTY15}).

By definition, a submanifold with principal curvatures in any direction occurring in pairs of opposite signs is called an \emph{austere
submanifold} (\cite{HL82}). Clearly, austere submanifolds are minimal submanifolds.
It is worth mentioning that \cite{QT16} also constructed two sequences of totally isoparametric, austere hypersurfaces by using the Clifford algebra.
However, our examples in Proposition \ref{prop} are different from theirs.

The paper is organized as follows.
In Section \ref{sec-ineq}, we prove the normal scalar curvature inequalities (\ref{1st}, \ref{2nd}) and characterize the equality conditions, i.e.,
the sets $C_A$, $C_P$ and $C_E$, by using formulae (\ref{sff}, \ref{normalscalar}), the Simons identity (\cite{Sim68}) and some identities in \cite{OT75}.
Section \ref{sec-C_A} is devoted to the classification of the set $C_A$ where Condition A is satisfied.
The algebra of Cayley numbers (Octonions), the Clifford algebra and the techniques in \cite{FKM81} will play an important role in the proof.
In Sections \ref{sec-C_P} and \ref{sec-C_E}, we determine the sets $C_P$
and $C_E$.
There we will use the methods in \cite{TY15}, \cite{QTY13} and \cite{LZ16} where the parallelism and Einstein conditions were considered globally.
Lastly, in Section \ref{sec-prop1.2}, we prove Proposition \ref{prop},
where the relations between the focal submanifolds with $g=4$, $(8,7)$ are shown and totally isoparametric functions are constructed.

\section{Proof of the normal scalar curvature inequalities}\label{sec-ineq}

In this section, we aim at giving a proof of the normal scalar curvature inequalities (\ref{1st}, \ref{2nd}) and their equality characterizations on focal submanifolds of isoparametric hypersurfaces with $g=4$ in unit spheres.

Given a point of the focal submanifold $M_+^{m_1+2m_2}$ in the unit sphere $S^{2m_1+2m_2+1}(1)$, as we stated in the introduction, choosing a unit normal frame $\{n_0, n_1, \ldots, n_{m_1}\}$ of $M_+$, we can
express the corresponding shape operators $\{S_0, S_1,\ldots, S_{m_1}\}$ to be the form of (\ref{sff}).
Then a direct calculation leads to
\begin{equation}\label{S0Sa}
[S_0, S_{\alpha}]=S_0S_{\alpha}-S_{\alpha}S_0=\left(
  \begin{array}{ccc}
    0 & 2A_{\alpha} & B_{\alpha} \\
    -2A_{\alpha}^t & 0 & -C_{\alpha} \\
    -B_{\alpha}^t & C_{\alpha}^t & 0 \\
  \end{array}
\right),\quad 1\leq\alpha\leq m_1.
\end{equation}
By the definition $\langle S_{\alpha}, S_{\beta}\rangle=:\mathrm{Trace}(S_{\alpha}S_{\beta}^t)$, (\ref{S0Sa}) implies that
\begin{equation}\label{commutator}
  \|[S_0, S_{\alpha}]\|^2=8\|A_{\alpha}\|^2+2\|B_{\alpha}\|^2+2\|C_{\alpha}\|^2.
\end{equation}
On the other hand, as we pointed out before, the shape operator of any unit normal vector has eigenvalues $+1$, $-1$ and $0$
with multiplicities $m_2$, $m_2$ and $m_1$, respectively.
It follows immediately that $\langle S_{\alpha}, S_{\beta}\rangle=2m_2\delta_{\alpha\beta}$, in particular,
\begin{equation}\label{shapeopnorm}
\|S_{\alpha}\|^2=2(\|A_{\alpha}\|^2+\|B_{\alpha}\|^2+\|C_{\alpha}\|^2)=2m_2, \quad 1\leq\alpha\leq m_1.
\end{equation}
It follows from (\ref{commutator}) and (\ref{shapeopnorm}) that
\begin{equation*}
 2m_1m_2\leq\sum_{\alpha=1}^{m_1}\|[S_0, S_{\alpha}]\|^2\leq 8m_1m_2.
\end{equation*}
Here the first equality holds if and only if $A_{\alpha}=0$, $\forall ~1\leq \alpha\leq m_1$, and the second
holds if and only if $B_{\alpha}=C_{\alpha}=0$, $\forall~ 1\leq \alpha\leq m_1$.
Given any $\beta=0,\cdots, m_1$, we can diagonalize $S_\beta$ as $diag(I_{m_2}, -I_{m_2}, 0_{m_1})$, while the other $S_\alpha$'s ($\alpha\neq\beta$) are in the block-form of (\ref{sff}). Thus the same argument as above shows
\begin{equation*}
 2m_1m_2\leq\sum_{\alpha=0}^{m_1}\|[S_{\alpha}, S_{\beta}]\|^2\leq 8m_1m_2,
\end{equation*}
which implies the inequality (\ref{1st}) immediately:
\begin{equation*}
2m_1m_2(m_1+1)\leq \rho^{\bot}=:\sum_{\alpha,\beta=0}^{m_1}\|[S_{\alpha}, S_{\beta}]\|^2\leq 8m_1m_2(m_1+1).
\end{equation*}
Clearly, at a point where the upper bound of $\rho^{\bot}$ is achieved, the Condition A must be satisfied. Conversely,
it is not difficult to see that the upper bound of $\rho^{\bot}$ can be achieved at a point of Condition A.

We focus next on the first equality case in (\ref{1st}).
The Simons identity for minimal submanifolds $N$ in a unit sphere with codimension $p$ is stated as (cf. \cite{Sim68}, \cite{CdK})
\begin{equation}\label{Simons}
\frac{1}{2}\Delta\|S\|^2=\|\overline{\nabla}S\|^2+ \dim N\cdot\|S\|^2-\sum_{\alpha,\beta=1}^p\langle S_{\alpha}, S_{\beta}\rangle^2-\rho^{\bot},
\end{equation}
where we denote by $S$ the second fundamental form and $\overline{\nabla}S$ its covariant derivative.
In our case, on the focal submanifold $M_+$, $\dim M_+=m_1+2m_2$, $\|S\|^2=2m_2(m_1+1)$,
 the Simons identity (\ref{Simons}) reduces to
\begin{equation*}
  \|\overline{\nabla}S\|^2=\rho^{\bot}-2m_1m_2(m_1+1).
\end{equation*}
Consequently, at a point in $M_+$, the first equality in (\ref{1st}) holds if and only if the second fundamental form is parallel. We denoted the set of these points by $C_P$ in the introduction.

To derive the normal scalar curvature inequality (\ref{2nd}), we firstly recall some useful formulae in Lemma 12 of \cite{OT75}:
\begin{equation}\label{OT formula}
  S_{\alpha}=S_{\alpha}^3, \quad S_{\alpha}=S_{\beta}^2S_{\alpha}+S_{\alpha}S_{\beta}^2+S_{\beta}S_{\alpha}S_{\beta},\quad  \alpha\neq \beta.
\end{equation}
Using the formulae (\ref{shapeopnorm}, \ref{OT formula}), we are able to derive the new expression (\ref{another rho}) for the normal scalar curvature:
\begin{eqnarray*}
\rho^{\bot} &=:& \sum_{\alpha,\beta=0}^{m_1}\|[S_{\alpha}, S_{\beta}]\|^2= -\sum_{\alpha,\beta=0}^{m_1}\mathrm{Trace}(S_{\alpha}S_{\beta}-S_{\beta}S_{\alpha})^2 \\
   &=&  2\sum_{\alpha,\beta=0,\atop\alpha\neq \beta}^{m_1}\mathrm{Trace}(S_{\alpha}^2S_{\beta}^2-S_{\alpha}(S_{\alpha}-S_{\beta}^2S_{\alpha}-S_{\alpha}S_{\beta}^2))\\
   &=&  6\sum_{\alpha,\beta=0}^{m_1}\mathrm{Trace}(S_{\alpha}^2S_{\beta}^2)-6\sum_{\alpha=0}^{m_1}\mathrm{Trace}(S_{\alpha}^2)-2m_1\sum_{\alpha=0}^{m_1}\mathrm{Trace}(S_{\alpha}^2) \\
   &=&  6\|\sum_{\alpha=0}^{m_1}S_{\alpha}^2\|^2-4m_2(m_1+1)(m_1+3).
\end{eqnarray*}
As we stated before, $\mathrm{Trace}\sum_{\alpha=0}^{m_1}S_{\alpha}^2=\sum_{\alpha=0}^{m_1}\|S_{\alpha}\|^2=2m_2(m_1+1)$. By virtue of Schwartz inequality, we obtain $\|\mathrm{Trace}\sum_{\alpha=0}^{m_1}S_{\alpha}^2\|^2\leq \|\sum_{\alpha=0}^{m_1}S_{\alpha}^2\|^2\cdot\dim M_+$,
alternatively speaking,
$$\|\sum_{\alpha=0}^{m_1}S_{\alpha}^2\|^2\geq \frac{4m_2^2(m_1+1)^2}{m_1+2m_2}.$$
In this way, we get a new lower bound of $\rho^{\bot}$ as in the inequality (\ref{2nd}):
$$\rho^{\bot}\geq 2m_1m_2(m_1+1)+\frac{6m_1m_2(m_1+1)}{2m_2+m_1}(2m_2-m_1-2).$$
To see the equality case, taking advantage of the minimality of $M_+$ in $S^{n+1}(1)$ and using the Gauss equation, we can express the Ricci curvature of $X\in T_xM_+$ as:
\begin{equation}\label{Ric }
Ric(X)=(\dim M_+-1)|X|^2-\langle\sum_{\alpha=0}^{m_1}S_{\alpha}^2X, X\rangle.
\end{equation}
Therefore, the equality in Schwartz inequality holds, i.e., the equality in (\ref{2nd}) holds if and only if
the Ricci curvature is constant, i.e., the point belongs to the set $C_E$.
Obviously, if $2m_2-m_1-2>0$, the lower bound in (\ref{1st}) cannot be achieved, and the set $C_P$ is empty on $M_+$.

\section{The set $C_A$-Condition A}\label{sec-C_A}
 In this section we figure out the set $C_A$ in Table \ref{table-M+} and Table \ref{table-M-} where Condition A is satisfied, or equivalently, the normal scalar curvature $\rho^\perp$ achieves its upper bound in (\ref{1st}) on the focal submanifolds $M_{\pm}$. The proof proceeds according to the important theorem of \cite{DN83} mentioned in the introduction.

\subsection{$M_+$ with $(3, 4k)$ of OT-FKM type.}
~\vskip 0.5pt
Recall that the focal submanifold $M_+$ of OT-FKM type with multiplicities $(m, l-m-1)$ can be described as
\begin{equation}\label{M+}
  M_+=\{x\in S^{2l-1}(1)~|~\langle P_0x, x\rangle=\cdots=\langle P_mx, x\rangle=0\}.
\end{equation}
Now let $m=3$. As one knows, $\{P_0x,P_1x, P_2x, P_3x\}$ is an orthonormal basis of the normal space $T_x^{\bot}M_+$, and the corresponding shape operator is
$S_{\alpha}X=:S_{P_{\alpha}x}X=-(P_{\alpha}X)^{T}$, $\alpha=0,\cdots, 3$. More precisely,
\begin{equation*}
-S_0X =P_0X-\langle P_0X, x\rangle x-\sum_{\alpha=0}^3\langle P_0X, P_{\alpha}x\rangle P_{\alpha}x=P_0X- \sum_{\alpha=1}^3\langle X, P_0P_{\alpha}x\rangle P_{\alpha}x.
\end{equation*}
As a result,
\begin{equation*}
 S_0X=0 \Leftrightarrow P_0X\in \mathrm{Span}\{P_1x, P_2x, P_3x\}\Leftrightarrow X\in \mathrm{Span}\{P_0P_1x, P_0P_2x, P_0P_3x\}.
\end{equation*}
Analogously, by the geometric meaning of Condition A, i.e., the kernels of all shape operators coincide,  the arguments above imply that
\begin{eqnarray*}
x\in C_A &\Leftrightarrow&  ~~x\in M_+, \mathrm{and}\\
&&\mathrm{Span}\{P_0P_1x, P_0P_2x, P_0P_3x\}=\mathrm{Span}\{P_1P_0x, P_1P_2x, P_1P_3x\}\\
                          && =\mathrm{Span}\{P_2P_0x, P_2P_1x, P_2P_3x\}=\mathrm{Span}\{P_3P_0x, P_3P_1x, P_3P_2x\}.
\end{eqnarray*}
Evidently, $\langle P_0P_2x, P_1P_0x\rangle=\langle P_0P_2x, P_1P_2x\rangle=0$, and $\langle P_0P_3x, P_1P_0x\rangle=\langle P_0P_3x, P_1P_3x\rangle$ $=0$. Thus $\mathrm{Span}\{P_0P_1x, P_0P_2x, P_0P_3x\}=\mathrm{Span}\{P_1P_0x, P_1P_2x, P_1P_3x\}$ if and only if $P_0P_2x=\pm P_1P_3x$.
Analogously, we obtain a further sufficient and necessary condition of $x\in C_A$ as
\begin{equation*}
 x\in M_+~~\mathrm{and}~~ P_0P_1P_2P_3x=\pm x.
\end{equation*}
Obviously, the condition $P_0P_1P_2P_3x=\pm x$ guarantees $x\in M_+$, thus
$x\in C_A$ if and only if $x\in E_{\pm}(P_0P_1P_2P_3)\cap S^{8k+7}(1)=S^{4k+3}(1)\sqcup S^{4k+3}(1)$ (disjoint union of two connected components),
since the matrix $P_0P_1P_2P_3$ is symmetric, orthogonal with trace zero: $\mathrm{Trace}(P_0P_1P_2P_3)=\mathrm{Trace}(P_1P_2P_3P_0)=-\mathrm{Trace}(P_0P_1P_2P_3)$.
Here and throughout this paper, for a symmetric, orthogonal $2l\times 2l$ matrix $Q$ with vanishing trace, we denote by $E_{\pm}(Q)$ the eigenspaces associated with eigenvalues $\pm1$, which have dimension $l=m_1+m_2+1$, half the dimension of the Euclidean space $\mathbb{R}^{2l}$.
Up to now, we have proved:

\begin{prop}\label{3,4k}
On $M_+$ with $(3, 4k)$ of OT-FKM type, $C_A=S^{4k+3}(1)\sqcup S^{4k+3}(1)$.
\end{prop}
\vspace{2mm}

\subsection{$M_+$ with $(1,k)$, $(4,3)$ (OT-FKM type) and $(2,2)$.}
~\vskip 0.5pt
Firstly, we recall a well-known fact that the isoparametric families with multiplicities $(2,2)$ and OT-FKM type $(1,k)$, $(4,3)$ (definite)
are homogeneous. This explains the cases $C_A=M_{\pm}$ in Table \ref{table-M+} and \ref{table-M-} for $(1,k)$ and definite $(4, 3)$ case, since the normal scalar curvature is constant in homogeneous cases.

To determine which one of, or whether both of the focal submanifolds satisfy Condition A in the $(2,2)$ case, we recall the normal scalar curvature inequalities (\ref{1st}, \ref{2nd}). Notice that now the two lower bounds in (\ref{1st}, \ref{2nd}) coincide with each other, i.e., the parallel condition (set $C_P$) coincides with the Einstein condition (set $C_E$). Notice also that the lower bound and the upper bound in (\ref{1st}) are different, thereby only one bound can be achieved on one single focal submanifold.
Fortunately, it was proved in \cite{QTY13} that the focal submanifold (say $M_-$) diffeomorphic to the oriented Grassmannian $\widetilde{G}_2(\mathbb{R}^5)$ is Einstein, while the other focal submanifold (say $M_+$) diffeomorphic to $\mathbb{C}P^3$ is not. Hence $C_P=C_E=M_-$ and $C_A=\emptyset$ on $M_-$. On the other hand, \cite{OT76} and \cite{DN83} asserted that there must be a point satisfying Condition A on focal submanifolds. Therefore, $C_A=M_+$ and $C_P=C_E=\emptyset$ on $M_+$.

Besides, as we mentioned in Remark \ref{rem1}, the indefinite $(4, 3)$ family is congruent to the $(3, 4)$ family in OT-FKM type. Thus the set $C_A$ in $M_-$ of the indefinite $(4, 3)$ family
coincides with that in $M_+$ of the $(3, 4)$ family. Namely, $C_A=S^7(1)\sqcup S^7(1)$, which was proved in Subsection 3.1. We thus get the following
\begin{prop}\label{221k43}
On the focal submanifolds with $(1,k)$, $(4,3)$ (OT-FKM type) and $(2,2)$, we have
\begin{itemize}
\item [(1)]On $M_+$ with $(1,k)$, $C_A=M_+$;
\item [(2)]On $M_-$ with $(4, 3)$ of OT-FKM type , in the definite case, $C_A=M_-$, and in the indefinite case, $C_A=S^7(1)\sqcup S^7(1)$;
\item [(3)]In the $(2,2)$ case, on the focal submanifold $M_+$ diffeomorphic to $\mathbb{C}P^3$, $C_A=M_+$ and $C_P=C_E=\emptyset$; while on the focal submanifold $M_-$ diffeomorphic to the oriented Grassmannian $\widetilde{G}_2(\mathbb{R}^5)$, $C_P=C_E=M_-$ and $C_A=\emptyset$.
\end{itemize}
\end{prop}

\subsection{$M_+$ with $(7, 8k)$ of OT-FKM type.}\label{Sec4}
~\vskip 0.5pt
Given a Clifford system $\{P_0,\cdots, P_7\}$ on $R^{16k+16}$, we define $P=:P_0\cdots P_7$. It is straightforward to verify that
$P$ is symmetric with vanishing trace and $P^2=Id$.
We shall start by establishing an important assertion.
\begin{assert}\label{assert1}
 $Px=\pm x$ for $x\in C_A\subset M_+$.
\end{assert}

\begin{proof}
Using a similar argument as in Subsection 3.1, we find an equivalent condition of $x\in C_A$ as:
\begin{eqnarray}\label{7 8k}
&&~~\mathrm{Span}\{P_0P_{\alpha}x, \alpha=0,\cdots,7, \alpha\neq0\}=\mathrm{Span}\{P_1P_{\alpha}x, \alpha=0,\cdots,7, \alpha\neq1\} \\
&&= \cdots=\mathrm{Span}\{P_7P_{\alpha}x, \alpha=0,\cdots,7, \alpha\neq7\}.\nonumber
\end{eqnarray}
With the known fact that $\langle x, P_{\alpha}P_{\beta}x\rangle=0$ for any $\alpha, \beta=0,\cdots, 7, \alpha\neq\beta$, we can interpret (\ref{7 8k})
in the following way:
\begin{eqnarray}\label{7 8k'}
&&\mathrm{Span}\{P_0P_{\alpha}x, \alpha=0,\cdots,7\}=\mathrm{Span}\{P_1P_{\alpha}x, \alpha=0,\cdots,7\} \\
&&= \cdots=\mathrm{Span}\{P_7P_{\alpha}x, \alpha=0,\cdots,7\}.\nonumber
\end{eqnarray}
In other words, if $x\in C_A$, then the space $\mathrm{Span}\{P_{\alpha}P_0x, \cdots, P_{\alpha}P_7x\}$ is independent of the choice of $\alpha$.
Thus we can define $V=:\mathrm{Span}\{P_{\alpha}P_0x, \cdots, P_{\alpha}P_7x\}$. Then it follows that
\begin{eqnarray*}
 V=\mathrm{Span}\{P_{\beta}P_0x, \cdots, P_{\beta}P_7x\}&\xrightarrow{P_{\beta}}&\mathrm{Span}\{P_0x, \cdots, P_7x\}\\
 &\xrightarrow{P_{\alpha}}&\mathrm{Span}\{P_{\alpha}P_0x, \cdots, P_{\alpha}P_7x\}=V,
\end{eqnarray*}
that is, $P_{\alpha}P_{\beta}: V\rightarrow V$, $\forall~ \alpha, \beta=0,\cdots, 7$.
Since $x\in V$, we get $Px=P_0\cdots P_7x\in V$. But we notice that for $\alpha\neq 0$,
$\langle Px, P_0P_{\alpha}x\rangle=(-1)^{\alpha-1}\langle P_1\cdots \widehat{P_{\alpha}}\cdots P_7x, x\rangle=0$, because $P_1\cdots \widehat{P_{\alpha}}\cdots P_7$
is anti-symmetric. It leaves only that $Px=\pm x$, as desired.
\end{proof}

We remark that if $Px=\pm x$ for $x\in S^{16k+15}(1)$, then $\langle P_{\alpha}x, x\rangle=\pm \langle P_{\alpha}x, P_0\cdots P_7x\rangle=0$ for any $\alpha=0,\cdots,7$, thus $x\in M_+$.

Next, let us represent $x\in S^{16k+15}(1)$ by $x=(u,v)$ with $|u|^2+|v|^2=1$ and
$$u=(u_1,\cdots,u_{k+1})\in\mathbb{O}^{k+1},  ~v=(v_1,\cdots,v_{k+1})\in\mathbb{O}^{k+1}.$$
We give now a specific expression of the Clifford system on $\mathbb{R}^{16k+16}\cong\mathbb{O}^{2k+2}$, as there exists exactly
one algebraic equivalence class of Clifford system when $m=7\not\equiv0(\mathrm{mod} ~4)$:
\begin{equation}\label{clifford system}
P_0(u ,v)=(u, -v),\quad P_{\alpha}(u, v)=(E_{\alpha}v, -E_{\alpha}u), \alpha=1,\cdots,7,
\end{equation}
where $E_{\alpha}$ acts on $u$ or $v$ in this way:
$$E_{\alpha}u=(e_{\alpha}u_1,\cdots,e_{\alpha}u_{k+1}),~\alpha=1,\cdots,7,$$
and $\{1,e_1,e_2,\cdots, e_7\}$ is the standard orthonormal basis of the octonions (Cayley numbers) $\mathbb{O}$.

Let $\mathrm{Im}\mathbb{O}$ be the set of purely imaginary Cayley numbers, and $wu=:(wu_1,\cdots, wu_{k+1})$. Noticing that $P_0P_{\alpha}x=(E_{\alpha}v, E_{\alpha}u)$ ($\alpha\neq 0$), we can illustrate (\ref{7 8k}) in a simpler way as
\begin{assert}\label{assert2}
$x\in C_A$ if and only if $x\in M_+$ and $P_{\alpha}P_{\beta}x\in \{(wv, wu)~|~w\in \mathrm{Im}\mathbb{O}\}$, for $\alpha,\beta=1,...,7,\alpha\neq\beta.$
\end{assert}
\begin{assert}\label{assert3}
$P(u, v)=(v, u)$, for $u, v\in\mathbb{O}^{k+1}$.
\end{assert}

\begin{proof}
Let $1=(1, 0), e_1=(i, 0), e_2=(j, 0), e_3=(k, 0), e_4=(0, 1), e_5=(0, i), e_6=(0, j), e_7=(0, k)\in\mathbb{H}\times\mathbb{H}$.
Recalling the Cayley-Dickson construction of the product of Octonions $\mathbb{O}\cong \mathbb{H}\times\mathbb{H}$:
\begin{eqnarray*}
\mathbb{O}\times \mathbb{O}&\longrightarrow&\mathbb{O}\\
(a,b),~(c,d)&\mapsto&(a,b)\cdot (c,d) =: (ac-\bar{d}b, ~ da+b\bar{c}),
\end{eqnarray*}
 one can see easily that $e_1(e_2(\cdots(e_7x)))=-x$, $\forall~ x\in\mathbb{O}$. Then the conclusion
follows immediately from $P_0\cdots P_7(u,v)=(-E_1(\cdots(E_7v)), -E_1(\cdots(E_7u)))$.
\end{proof}

For $x\in C_A$, we can conclude from Assertions \ref{assert1} and \ref{assert3} that $u=\pm v$.
Let us focus on the case $u=v$, i.e., $x=(u , u)\in S^{16k+15}(1)$. The argument in the case $u=-v$ is analogous.
Then Assertion \ref{assert2} becomes
\begin{eqnarray*}
&&x=(u, u)\in C_A \\
&\Leftrightarrow& P_{\alpha}P_{\beta}x=(-E_{\alpha}(E_{\beta}u), -E_{\alpha}(E_{\beta}u))\in\{(w'u, w'u)~|~w'\in \mathrm{Im}\mathbb{O}\}, \forall~\alpha\neq\beta\geq 1, \\
 &\Leftrightarrow& E_{\alpha}(E_{\beta}u)=(e_{\alpha}(e_{\beta}u_1),\cdots,e_{\alpha}(e_{\beta}u_{k+1}))=(wu_1,\cdots, wu_{k+1}) ~\mathrm{for~some}~w\in \mathrm{Im}\mathbb{O},\\
&& \forall~\alpha\neq\beta\geq 1,  \\
 &\Leftrightarrow& \frac{1}{|u_i|^2}(e_{\alpha}(e_{\beta}u_i))\bar{u_i}=\frac{1}{|u_j|^2}(e_{\alpha}(e_{\beta}u_j))\bar{u_j}=w\in \mathrm{Im}\mathbb{O}~for~u_i, u_j\neq 0, i,j=1,\cdots,k+1,\\
&& ~\mathrm{for~some}~w\in \mathrm{Im}\mathbb{O}, \forall~\alpha\neq\beta\geq 1.
\end{eqnarray*}
As a matter of fact, we notice that $\mathrm{Re} ~(e_{\alpha}(e_{\beta}u_i))\bar{u_i}=\langle e_{\alpha}(e_{\beta}u_i), u_i\rangle=\langle e_{\beta}u_i, \bar{e}_{\alpha}u_i\rangle=\langle e_{\beta}, -e_{\alpha}\rangle|u_i|^2=0$, i.e., $(e_{\alpha}(e_{\beta}u_i))\bar{u_i}\in \mathrm{Im}\mathbb{O}$, for $\alpha\neq \beta\geq 1$.
To continue the investigation, we need

\begin{defn}
For two unit Cayley numbers $\sigma$, $\tau$, the pair $\{\sigma, \tau\}$ satisfies Condition X, if $(e_{\alpha}(e_{\beta}\sigma))\bar{\sigma}=(e_{\alpha}(e_{\beta}\tau))\bar{\tau}$, for any $\alpha, \beta=1,\cdots, 7, \alpha\neq \beta$.
\end{defn}
Clearly, Condition X is equivalent to that $(x(y\sigma))\bar{\sigma}=(x(y\tau))\bar{\tau}$, $\forall~x, y\in \mathbb{O}$.

\begin{defn}
For two unit Cayley numbers $\sigma$, $\tau$, the pair $\{\sigma, \tau\}$ satisfies Condition Y, if
$(x\sigma)\tau=x(\sigma\tau)$, for any $x\in \mathbb{O}$.
\end{defn}

Now we would like to list three observations:

Taking $y=\bar{\sigma}$ in the definition of Condition X, we obtain easily the following:
\begin{observ}\label{observ1}
 For $\sigma, \tau\in \mathbb{O}$ with $|\sigma|=|\tau|=1$, suppose $\{\sigma, \tau\}$ satisfies Condition X.
Then $\{\bar{\sigma}, \tau\}$ satisfies Condition Y.
\end{observ}

Writing $\sigma=(\sigma_1, \sigma_2)\in \mathbb{H}\times\mathbb{H}$, $\tau=(\tau_1, \tau_2)\in \mathbb{H}\times\mathbb{H}$, we obtain
\begin{observ}\label{observ2}
 For $\sigma, \tau\in \mathbb{O}$ with $|\sigma|=|\tau|=1$, suppose $\{\sigma, \tau\}$ satisfies Condition X.
Then $|\sigma_1|=|\tau_1|$, $|\sigma_2|=|\tau_2|$.
\end{observ}

\begin{proof}
Take $x=(i, 0)$, $y=(j, 0)$. Then
\begin{eqnarray*}
(x(y\sigma))\bar{\sigma} &=& (ij\sigma_1, \sigma_2ji)\bar{\sigma}=(k\sigma_1, -\sigma_2k)(\bar{\sigma}_1, -\sigma_2)=((|\sigma_1|^2-|\sigma_2|^2)k, -2\sigma_2k\sigma_1), \\
(x(y\tau))\bar{\tau} &=& (ij\tau_1, \tau_2ji)\bar{\tau}=(k\tau_1, -\tau_2k)(\bar{\tau}_1, -\tau_2)=((|\tau_1|^2-|\tau_2|^2)k, -2\tau_2k\tau_1).
\end{eqnarray*}
By the assumption that $\{\sigma, \tau\}$ satisfies Condition X, we derive from the first components of $(x(y\sigma))\bar{\sigma}$ and $(x(y\tau))\bar{\tau}$ that
$|\sigma_1|^2-|\sigma_2|^2=|\tau_1|^2-|\tau_2|^2$. Combining with $|\sigma|=|\tau|=1$, it follows easily that $|\sigma_1|=|\tau_1|$, $|\sigma_2|=|\tau_2|$.
\end{proof}

\begin{observ}\label{observ3}
 For unit $\sigma=(\sigma_1, \sigma_2)\in \mathbb{O}$ with $\sigma_2\neq 0$, and unit $\tau=(\tau_1, \tau_2)\in \mathbb{O}$,
the pair $\{\sigma, \tau\}$ satisfies Condition Y if and only if $\mathrm{Im}~ \sigma \sslash  \mathrm{Im}~ \tau$, where $\mathrm{Im}~ \sigma$ and $\mathrm{Im}~ \tau$ stand for the imaginary parts of $\sigma$ and $\tau$, respectively.
\end{observ}

\begin{proof}
Take an arbitrary Cayley number $x=(x_1, x_2)\in \mathbb{H}\times\mathbb{H}.$ We calculate
\begin{eqnarray*}
(x\sigma)\tau &=& (x_1\sigma_1-\bar{\sigma}_2x_2, \sigma_2x_1+x_2\bar{\sigma}_1)(\tau_1, \tau_2)\\
&=&((x_1\sigma_1-\bar{\sigma}_2x_2)\tau_1-\bar{\tau}_2(\sigma_2x_1+x_2\bar{\sigma}_1), \tau_2(x_1\sigma_1-\bar{\sigma}_2x_2)+(\sigma_2x_1+x_2\bar{\sigma_1})\bar{\tau}_1) \\
x(\sigma\tau) &=& (x_1, x_2)(\sigma_1\tau_1-\bar{\tau}_2\sigma_2, \tau_2\sigma_1+\sigma_2\bar{\tau}_1)\\
&=& (x_1(\sigma_1\tau_1-\bar{\tau}_2\sigma_2)-\overline{\tau_2\sigma_1+\sigma_2\bar{\tau}_1}x_2, (\tau_2\sigma_1+\sigma_2\bar{\tau}_1)x_1+x_2\overline{\sigma_1\tau_1-\bar{\tau}_2\sigma_2}).
\end{eqnarray*}
Then the pair $\{\sigma, \tau\}$ satisfies Condition Y, i.e., $(x\sigma)\tau=x(\sigma\tau)$, for arbitrary $x=(x_1, x_2)\in \mathbb{H}\times\mathbb{H}$, if and only if
\begin{eqnarray}
&&\bar{\tau}_2\sigma_2x_1=x_1\bar{\tau}_2\sigma_2, \quad \forall x_1\in\mathbb{H}, \label{1}\\
&& \bar{\sigma}_2x_2\tau_1+\bar{\tau}_2x_2\bar{\sigma}_1=\bar{\sigma}_1\bar{\tau}_2x_2+\tau_1\bar{\sigma}_2 x_2, \quad \forall x_2\in\mathbb{H}, \label{2}\\
&&\tau_2x_1\sigma_1+\sigma_2x_1\bar{\tau}_1=\tau_2\sigma_1x_1+\sigma_2\bar{\tau}_1x_1, \quad \forall x_1\in\mathbb{H}, \label{3}\\
&&x_2\bar{\sigma}_1\bar{\tau}_1-\tau_2\bar{\sigma}_2x_2=x_2\bar{\tau}_1\bar{\sigma}_1-x_2\bar{\sigma}_2\tau_2, \quad \forall x_2\in\mathbb{H}, \label{4}
\end{eqnarray}
Recall the fact that a quaternionic number which commutes with
all quaternionic numbers must be a real number. Thus
$$(\ref{1}) \Leftrightarrow \tau_2=\lambda\sigma_2, ~\mathrm{for~ some~ }\lambda\in\mathbb{R}.$$
Substituting $\tau_2=\lambda\sigma_2$ into (\ref{2}), we get $(\bar{\sigma}_2x_2)(\tau_1+\lambda\bar{\sigma}_1)=(\tau_1+\lambda\bar{\sigma}_1)(\bar{\sigma}_2x_2)$, $\forall~ x_2\in \mathbb{H}$. Notice that $\bar{\sigma}_2x_2$ achieves all the quaternionic numbers, since $\sigma_2\neq 0$. Thus
$$(\ref{1}), (\ref{2}) \Leftrightarrow  \tau_2=\lambda\sigma_2, ~\mathrm{for ~some}~ \lambda\in \mathbb{R}, ~\mathrm{and}~ \mu=:\tau_1+\lambda\bar{\sigma}_1\in\mathbb{R}. $$
Therefore, if (\ref{1}), (\ref{2}) hold, then (\ref{3}) and (\ref{4}) are satisfied automatically.
Summarizing the above, we have arrived at the conclusion as desired, that
\begin{eqnarray*}
&&\{\sigma, \tau\}~\mathrm{with}~\sigma_2\neq 0~ \mathrm{satisfies~ Condition~ Y}\\
&\Leftrightarrow& (\tau_1, \tau_2)=(\mu-\lambda\bar{\sigma}_1, \lambda\sigma_2), ~\mathrm{for ~some}~\lambda,\mu\in\mathbb{R} \\
&\Leftrightarrow& \mathrm{Im} \sigma \sslash  \mathrm{Im} \tau.
\end{eqnarray*}
\end{proof}

With all these preparations, we are in a position to characterize those points satisfying Condition X.
\begin{lem}\label{Cond X}
For unit Cayley numbers $\sigma, \tau\in \mathbb{O}$, the pair $\{\sigma, \tau\}$ satisfies Condition X, if and only if $\sigma=\pm\tau$.
\end{lem}

\begin{proof} Clearly, the pair $\{\sigma, \tau\}=\{\sigma, \pm\sigma\}$ satisfies Condition X. So we need only to prove the opposite side.

Suppose $\{\sigma, \tau\}$ satisfies Condition X. Expressing $\sigma, \tau\in \mathbb{O}$ as $\sigma=(\sigma_1, \sigma_2), \tau=(\tau_1, \tau_2)\in \mathbb{H}\times\mathbb{H}$, we divide the proof into two cases:

\noindent
\textbf{Case 1:} $\sigma_2=0$. By Observation \ref{observ2}, we have $|\tau_1|=|\sigma_1|$, $|\tau_2|=|\sigma_2|=0$, thus $\tau_2=0$. For any $x=(x_1, x_2), y=(y_1, y_2)\in \mathbb{H}\times\mathbb{H}$,
\begin{eqnarray*}
  (x(y\sigma))\bar{\sigma} &=& (x_1y_1|\sigma_1|^2-\sigma_1\bar{y}_1x_2\bar{\sigma}_1,~ y_1\bar{\sigma}_1x_1\sigma_1+x_2\bar{\sigma}_1\bar{y}_1\sigma_1), \\
  (x(y\tau))\bar{\tau} &=& (x_1y_1|\tau_1|^2-\tau_1\bar{y}_1x_2\bar{\tau}_1,~ y_1\bar{\tau}_1x_1\tau_1+x_2\bar{\tau}_1\bar{y}_1\tau_1).
\end{eqnarray*}
Thus $(x(y\sigma))\bar{\sigma}=(x(y\tau))\bar{\tau}$, $\forall~ x, y\in\mathbb{O}$ if and only if
\begin{eqnarray}\label{case1}
&& \sigma_1\bar{y}_1x_2\bar{\sigma}_1=\tau_1\bar{y}_1x_2\bar{\tau}_1, \\
&& \bar{\sigma}_1x_1\sigma_1=\bar{\tau}_1x_1\tau_1, ~~\bar{\sigma}_1\bar{y}_1\sigma_1=\bar{\tau}_1\bar{y}_1\tau_1, \forall~ x_1, x_2, y_1\in\mathbb{H}.\nonumber
\end{eqnarray}
From the arbitrary choices of $x_1, x_2$ and $y_1$, it follows that (\ref{case1}) is equivalent to
$\sigma_1z\bar{\sigma}_1=\tau_1z\bar{\tau}_1$ and $\bar{\sigma}_1z'\sigma_1=\bar{\tau}_1z'\tau_1$, for any $z, z'\in \mathbb{H}$, that is, $\bar{\sigma}_1\tau_1\in\mathbb{R}$.
Alternatively speaking, $\tau_1=\pm\sigma_1$, thus $\tau=\pm\sigma$.

\noindent
\textbf{Case 2:} $\sigma_2\neq0$. From Observations \ref{observ1} and \ref{observ3}, it follows that $\mathrm{Im} \bar{\sigma}\sslash \mathrm{Im} \tau$. Then combining with Observation \ref{observ2}, we obtain that
$$(\sigma, \tau)=(\sigma, \pm\sigma) ~or~(\sigma, \pm\bar{\sigma}).$$
Obviously, the pair $\{\sigma, \pm\sigma\}$ satisfies Condition X, so we are left to consider the case $\tau=\pm\bar{\sigma}$.
If $\mathrm{Re}~\sigma=0$, we have $\bar{\sigma}=-\sigma$, thus the pair $\{\sigma, \tau\}=\{\sigma, \pm\bar{\sigma}\}=\{\sigma, \mp\sigma\}$ satisfies Condition X.
If $\mathrm{Re}~\sigma\neq0$, we have $\sigma+\bar{\sigma}=2\mathrm{Re} \sigma=:\lambda\neq0\in\mathbb{R}$.
Here we take $\tau=\bar{\sigma}$, since the proof for $\tau=-\bar{\sigma}$ is analogous. A direct calculation leads to
\begin{eqnarray*}
(x(y\bar{\sigma}))\sigma&=&(x(y(\lambda-\sigma)))(\lambda-\bar{\sigma})=(\lambda xy-x(y\sigma))(\lambda-\bar{\sigma})\\
&=&\lambda^2xy-\lambda(xy)\bar{\sigma}-\lambda x(y\sigma)+(x(y\sigma))\bar{\sigma}, ~\mathrm{for}~ x, y\in\mathbb{O}.
\end{eqnarray*}
Thus the pair $\{\sigma, \tau\}=\{\sigma, \bar{\sigma}\}$ satisfies Condition X, if and only if
$$\lambda xy-(xy)\bar{\sigma}- x(y\sigma)=0,~\forall~ x, y\in\mathbb{O},$$
which is equivalent to $$\lambda xy+x(y\bar{\sigma})=(xy)\bar{\sigma}+x(y(\sigma+\bar{\sigma}))=(xy)\bar{\sigma}+\lambda xy, ~\forall~ x, y\in\mathbb{O}.$$
Namely, $$x(y\bar{\sigma})=(xy)\bar{\sigma}, ~\forall~ x, y\in\mathbb{O}.$$
However, for $\sigma_2\neq 0$, there do exist $x, y\in\mathbb{O}$ such that $x(y\bar{\sigma})\neq(xy)\bar{\sigma}$.
For example, let $x=(i, 0)$ and $y=(j, 0)$. Then $x(y\bar{\sigma})=(k\bar{\sigma}_1, \sigma_2k)$, while $(xy)\bar{\sigma}=(k\bar{\sigma}_1, -\sigma_2k)$.
Therefore, the pair $\{\sigma, \tau\}=\{\sigma, \bar{\sigma}\}$ does not satisfy Condition X.

Summarizing all the results above, if the pair $\{\sigma, \tau\}$ satisfies Condition X, we can obtain $\tau=\pm\sigma,$ as desired.
\end{proof}

Recall that by Assertions \ref{assert1} and \ref{assert2}, $x=(u, v)\in C_A$ if and only if $u=\pm v$, and for any two non-vanishing components $u_i$, $u_j$ of $u$,
the pair $\{\frac{u_i}{|u_i|}, \frac{u_j}{|u_j|}\}$ satisfies Condition X.
Thus Lemma \ref{Cond X} yields that $\frac{u_i}{|u_i|}=\pm \frac{u_j}{|u_j|}$, which leads to the following corollary:

\begin{cor}
On the focal submanifold $M_+$ of OT-FKM type with $(7, 8k)$, $x=(u, v)\in C_A$ if and only if there exist $\sigma\in\mathbb{O}$ with $|\sigma|=1$, and
$\lambda_1,\cdots, \lambda_{k+1}\in \mathbb{R}$ with $\sum_{i=1}^{k+1}\lambda_i^2=1$, such that $u=\pm v=\frac{1}{\sqrt{2}}(\lambda_1\sigma, \cdots, \lambda_{k+1}\sigma)$.
\end{cor}

\begin{prop}\label{7,8k}
$C_A$ of the focal submanifold $M_+$ of OT-FKM type with $(7, 8k)$ is isometric to
$$(S^k(1)\times S^7(1))/\mathbb{Z}_2\sqcup(S^k(1)\times S^7(1))/\mathbb{Z}_2,$$
where we identify $((\lambda_1, \cdots, \lambda_{k+1}), \sigma)$ with $((-\lambda_1, \cdots, -\lambda_{k+1}), -\sigma)$ in $S^k(1)\times S^7(1)$.
\end{prop}

\begin{proof}
We denote $C_A$ by $C_A=C_A^+\sqcup C_A^-$, where $C_A^+$ is the component with $x=(u, u)$ and $C_A^-$ is the component with $x=(u, -u)$.
As in \cite{TY13}, we define a map
\begin{eqnarray*}
\Psi: S^k(1)\times S^7(1) &\rightarrow& C_A^+\subset \mathbb{R}^{16k+16}\\
(\lambda_1, \cdots, \lambda_{k+1}), \sigma &\mapsto& \frac{1}{\sqrt{2}}(\lambda_1\sigma, \cdots, \lambda_{k+1}\sigma, \lambda_1\sigma, \cdots, \lambda_{k+1}\sigma)
\end{eqnarray*}
It satisfies $\Psi((\lambda_1, \cdots, \lambda_{k+1}), \sigma)=\Psi((-\lambda_1, \cdots, -\lambda_{k+1}), -\sigma)$.
The verification that $\Psi$ is an isometry is straightforward.
\end{proof}

\begin{rem}
Let $\eta$ be the Hopf line bundle over the real projective space $\mathbb{R}P^7$. As is well known, the Grothendieck ring $\widetilde{KO}(\mathbb{R}P^7)$
is cyclic of order $8$ with generator $\eta-\textbf{1}$ (\cite{Hus75}). Observe that $(S^k\times S^7)/\mathbb{Z}_2$ (compare \cite{TXY12}) is diffeomorphic to the
total space of the sphere bundle of $(k+1)\eta$. Hence $(S^k\times S^7)/\mathbb{Z}_2$ is diffeomorphic to $S^k\times\mathbb{R}P^7$ provided that $8$ divides
$k+1$.
\end{rem}

\subsection{$M_-$ with $(8, 7)$ of OT-FKM type.}\label{subsec-87}
~\vskip 0.5pt
In this subsection, we want to describe the set $C_A$ in $M_-$ of OT-FKM type with $(8, 7)$.

By the definition of the isoparametric polynomial (\ref{FKM isop. poly.}), following \cite{FKM81}, one has
\begin{eqnarray*}
M_-&=&F^{-1}(-1)\cap S^{31}(1) =  \{x\in S^{31}(1)~|~\sum_{\alpha=0}^8\langle P_{\alpha}x, x\rangle^2=1\}\\
   &=& \{x\in S^{31}(1)~|~\mathrm{there~exists}~Q_0\in\Sigma(P_0,\cdots,P_8) ~\mathrm{with} ~Q_0x=x\},
\end{eqnarray*}
where $\Sigma(P_0,\cdots, P_8)$ is the unit sphere in $\mathrm{Span}\{P_0,\cdots,P_8\}$, which is called the Clifford sphere.
Now given $x \in M_-$ and $Q_0\in \Sigma(P_0,\cdots,P_8)$ with $Q_0x=x$, we define
$$\Sigma_{Q_0}=:\{Q\in \Sigma(P_0,\cdots,P_8)|\  \langle Q_0, Q \rangle=:\frac{1}{2l}\mathrm{Trace}(Q_0Q)=0 \}, $$
which is the equatorial sphere of $\Sigma(P_0,\cdots,P_8)$ orthogonal to $Q_0$.
Then we can extend $Q_0$ to such a symmetric Clifford system $\{Q_0, Q_1, \cdots , Q_8\}$
with $Q_i\in\Sigma_{Q_0}$ $(i\geq 1)$ that $\Sigma(Q_0, Q_1, \cdots, Q_8)=\Sigma(P_0, P_1, \cdots, P_8)$.

Let us now choose $N_1, N_2, \cdots, N_{8}$ as an orthonormal basis of $T^{\bot}_xM_-$ in $S^{31}(1)$.
According to \textbf{4.5}(iii) of \cite{FKM81},
\begin{eqnarray}\label{normalker}
T_x^{\bot}M_-=\{N\in E_-(Q_0)~|~N\bot \Sigma_{Q_0}x\}, \\
  ker S_N=\{v\in E_+(Q_0)~|~v\bot x,~ v\bot \Sigma_{Q_0}N\},\nonumber
\end{eqnarray}
for any unit normal vector $N$. Besides,
Lemma 2.1 in \cite{TY15} provides an orthonormal basis of $T_xM_-$:
for a given $1\leq i\leq 8$,
\begin{equation}\label{o.n.b of TxM2}
\{Q_1x,\cdots,Q_8x,~Q_iN_1,\cdots,Q_iN_8,~~~Q_iQ_1x,\cdots,\widehat{Q_iQ_{i}x},\cdots,Q_iQ_8x\}.
\end{equation}
Let us arrange these basis vectors in an alternate order as eigenvectors of $Q_0$:
\begin{eqnarray}\label{E+-}
&&E_-(Q_0) = \mathrm{Span}\{N_1,\cdots, N_8,~Q_1x,\cdots,Q_8x\} \\
&&E_+(Q_0) = \mathrm{Span}\{Q_iN_1,\cdots, Q_iN_8,~Q_iQ_1x,\cdots,Q_iQ_8x\}~\mathrm{for~ a ~given}~ 1\leq i\leq 8\nonumber.
\end{eqnarray}
By the definition of $C_A$ and (\ref{normalker}), we find that $x\in C_A$ if and only if $\mathrm{Span}\{Q_1N,\cdots,Q_8N\}\subset E_+(Q_0)$ is independent of the choice of $N$.
Equivalently, for $x\in M_-$,
\begin{equation}\label{span}
   x\in C_A\Leftrightarrow \mathrm{Span}\{Q_iN_1,\cdots,Q_iN_8,~i=1,\cdots, 8\}~\mathrm{is~of~dimensional~8}.
\end{equation}
In addition, we observe that
$$\dim \mathrm{Span}\{Q_1N_k,\cdots,Q_8N_k\}=\dim \mathrm{Span}\{Q_iN_1,\cdots,Q_iN_8\}, ~\mathrm{for~ }~ 1\leq i, k\leq 8.$$
This leads to alternative equivalent conditions of Condition A:
\begin{eqnarray}
  x\in C_A &\Leftrightarrow& \mathrm{Span}\{Q_iN_1,\cdots,Q_iN_8\} ~\mathrm{is~ fixed~ for~ any}~ 1\leq i\leq 8, \label{CA}\\
           &\Leftrightarrow& \mathrm{Span}\{Q_iQ_1x,\cdots,Q_iQ_8x\} ~\mathrm{is~ fixed~ for~ any}~ 1\leq i\leq 8. \label{CA2}
\end{eqnarray}
Observe that
\begin{equation*}
\mathrm{Span}\{Q_jN_1,\cdots, Q_jN_8\}\xrightarrow{Q_j}\mathrm{Span}\{N_1,\cdots, N_8\}\xrightarrow{Q_i}\mathrm{Span}\{Q_iN_1,\cdots, Q_iN_8\}.
\end{equation*}
Thus for $x\in C_A\subset M_-$, $i,j,k=1,\cdots, 8, i\neq j$, we have
\begin{eqnarray}\label{Qinorm}
Q_iQ_j:&&~\mathrm{Span}\{Q_kN_1,\cdots, Q_kN_8\}=\mathrm{Span}\{Q_jN_1,\cdots, Q_jN_8\}\\
&&\rightarrow \mathrm{Span}\{Q_iN_1,\cdots, Q_iN_8\}=\mathrm{Span}\{Q_kN_1,\cdots, Q_kN_8\}.\nonumber
\end{eqnarray}
Then from (\ref{E+-}), (\ref{CA2}) and the fact that $Q_i$ maps $E_{\pm}(Q_0)$ to $E_{\mp}(Q_0)$, it follows that (\ref{Qinorm}) is equivalent to either of the following
\begin{equation}\label{QiQjx}
Q_iQ_j:  \mathrm{Span}\{Q_kQ_1x,\cdots, Q_kQ_8x\}\rightarrow \mathrm{Span}\{Q_kQ_1x,\cdots, Q_kQ_8x\},
\end{equation}
\begin{equation}\label{Qix}
Q_iQ_j:  \mathrm{Span}\{Q_1x,\cdots, Q_8x\}\rightarrow \mathrm{Span}\{Q_1x,\cdots, Q_8x\}.
\end{equation}
Conversely, it is evident that each of (\ref{Qinorm}-\ref{Qix}) implies (\ref{span}-\ref{CA2}) and hence, all of them are equivalent conditions for $x\in C_A$.

Now we are ready to determine the set $C_A\subset M_-$ for the case with $(8,7)$.
As in the introduction, we still add superscripts $I$ and $D$ to distinguish the corresponding focal submanifolds and the sets $C_A$
in the indefinite and definite cases, respectively. That is, the indefinite type Clifford system $\{P_0,\cdots,P_8\}$ for $I$ and the definite $\{\tilde{P_0}=:P_0P, \cdots, \tilde{P_8}=:P_8P\}$ for $D$. From now on, we denote $P=:P_0\cdots P_8\neq\pm Id$.
Firstly, we give the following
\begin{prop}\label{C_A F}
$C_A^I=M_-^I\cap M_-^D=(E_{+}(P)\cup E_-(P))\cap S^{31}(1)=S^{15}(1)\sqcup S^{15}(1).$
\end{prop}

\begin{proof} We shall start by verifying
\begin{equation}\label{M-FG}
  M_-^I\cap M_-^D=E_{\pm}(P)\cap S^{31}(1).
\end{equation}
If $x\in M_-^I$, then there exists $Q_0\in\Sigma(P_0,\cdots,P_8)$ such that $Q_0x=x$.
In addition, if $x\in M_-^D$, that is, 
$1=\sum_{\alpha=0}^8\langle \tilde{P}_{\alpha}x, x\rangle^2=\sum_{\alpha=0}^8\langle P_{\alpha}Px, x\rangle^2$,
which implies that $Px\in \mathrm{Span}\{P_0x,\cdots,P_8x\}=\mathrm{Span}\{Q_0x,\cdots,Q_8x\}$
by the aforementioned statement $\Sigma(Q_0, Q_1, \cdots, Q_8)=\Sigma(P_0, P_1, \cdots, P_8)$.
So we can express $Px=\sum_{\alpha=0}^8c_{\alpha}Q_{\alpha}x$ ($c_{\alpha}\in\mathbb{R}$) in the following way:
\begin{equation*}
  Px=PQ_0x=Q_0Px=c_0x+\sum_{\alpha=1}^8c_{\alpha}Q_0Q_{\alpha}x=c_0x-\sum_{\alpha=1}^8c_{\alpha}Q_{\alpha}x.
\end{equation*}
The arguments above imply that $Px=c_0x$ and thus $Px=\pm x$ since $|Px|=|x|=1$.
Therefore, $M_-^I\cap M_-^D\subset E_{\pm}(P)\cap S^{31}(1)$.

Conversely, suppose $x\in E_{+}(P)\cap S^{31}(1)$ ( the proof for $x\in E_{-}(P)\cap S^{31}(1)$ will follow the same way ).
We know that $P_{\alpha}^2=P^2=Id$, $\mathrm{Trace} P_{\alpha}=\mathrm{Trace} P=0$, which lead us to the
decomposition $\mathbb{R}^{32}=E_+(P_0)\oplus E_-(P_0)=E_+(P)\oplus E_-(P)$. Thus we can decompose $x$ as
$$x=x_++x_-\in (E_+(P)\cap E_+(P_0))\oplus(E_+(P)\cap E_-(P_0)).$$
If $x_+=0$ or $x_-=0$, it follows easily that $\langle P_0x, x\rangle^2=\langle P_0Px, x\rangle^2=1$, which implies $x\in M_-^I\cap M_-^D$. So we are left to consider $x_{+}\neq 0$ and $x_-\neq 0$.
Noticing that
\begin{equation*}
  P_{\alpha}: E_{\pm}(P_0)\rightarrow E_{\mp}(P_0),~~\alpha=1,\cdots,8,
\end{equation*}
we obtain $\dim E_+(P)\cap E_+(P_0)=\dim E_+(P)\cap E_-(P_0)=8$, and further,
$E_+(P)\cap E_-(P_0)=\mathrm{Span}\{P_1x_+,\cdots, P_8x_+\}$ for dimensional reason.
Thus we can express $x_-$ as
\begin{equation*}
  x_-=\sum_{\alpha=1}^8\langle x_-, P_{\alpha}\frac{x_+}{|x_+|}\rangle P_{\alpha}\frac{x_+}{|x_+|}=\frac{1}{|x_+|^2}\sum_{\alpha=1}^8\langle x_-, P_{\alpha}x_+\rangle P_{\alpha}x_+,
\end{equation*}
which implies that $|x_+|^2|x_-|^2=\sum_{\alpha=1}^8\langle x_-, P_{\alpha}x_+\rangle^2$.
Therefore,
\begin{eqnarray*}
\sum_{\alpha=0}^8\langle P_{\alpha}x, x\rangle^2 &=& (|x_+|^2-|x_-|^2)^2+\sum_{\alpha=1}^8\langle P_{\alpha}(x_++x_-), (x_++x_-)\rangle^2 \\
   &=& (|x_+|^2+|x_-|^2)^2-4|x_+|^2|x_-|^2+4 \sum_{\alpha=1}^8\langle x_-, P_{\alpha}x_+\rangle^2\\
   &=& |x|^4=1,
\end{eqnarray*} and
\begin{equation*}
  \sum_{\alpha=0}^8\langle \tilde{P}_{\alpha}x, x\rangle^2=\sum_{\alpha=0}^8\langle P_{\alpha}x, Px\rangle^2=\sum_{\alpha=0}^8\langle P_{\alpha}x, x\rangle^2=1,
\end{equation*}
that is, $x\in M_-^I\cap M_-^D$.

In conclusion, $M_-^I\cap M_-^D=E_{\pm}(P)\cap S^{31}(1)$, proving the assertion (\ref{M-FG}).

Next, we will prove 
\begin{equation}\label{CAindef}
C_A^I=E_{\pm}(P)\cap S^{31}(1).
\end{equation}
Suppose $x\in C_A^I
\subset M_-^I$. Then there exists $Q_0\in \Sigma(P_0,\cdots, P_8)$ such that $Q_0x=x$.
We extend $Q_0$ to a symmetric Clifford system $\{Q_0,\cdots, Q_8\}$ in such a way that $P=P_0\cdots P_8=Q_0\cdots Q_8$,
which makes $Px=Q_1\cdots Q_8x$. Since $x\in C_A^I$, we have (\ref{QiQjx}):
\begin{equation*}
Q_iQ_j:  \mathrm{Span}\{Q_kQ_1x,\cdots, Q_kQ_8x\}\rightarrow \mathrm{Span}\{Q_kQ_1x,\cdots, Q_kQ_8x\}, \quad 1\leq i,j,k\leq8, i\neq j.
\end{equation*}
Cycling down, we get $Px=(Q_1\cdots Q_6)(Q_7Q_8x)\in \mathrm{Span}\{Q_kQ_1x,\cdots, Q_kQ_8x\}$.
Then from the observation that $\langle Px, Q_kQ_ix\rangle=\langle Q_1\cdots Q_8x, Q_kQ_ix\rangle=0$, $\forall~ i\neq k$,
we derive $Px=Q_1\cdots Q_8x=\pm x$. Namely, $x\in E_{\pm}(P)\cap S^{31}(1)$.

Conversely, suppose $x\in E_{\pm}(P)\cap S^{31}(1)$. From (\ref{M-FG}), 
it follows that $x\in M_-^I\cap M_-^D\subset M_-^I$. Then there exist $Q_0\in \Sigma(P_0,\cdots, P_8)$ such that $Q_0x=x$,
which leads to $Q_1\cdots Q_8x=Px=\pm x$. In other words, $x\in E_+(Q_0)\cap E_{\pm}(Q_1\cdots Q_8)$.
Then following the arguments on p.498-499 of \cite{FKM81}, we obtain
$$ker S_N=\{\omega\in E_+(Q_0)\cap E_{\pm}(Q_1\cdots Q_8)~|~\langle \omega, x\rangle=0\},$$ which is a fixed set independent of the choice of the unit normal vector $N$.
Alternatively speaking, $E_{\pm}(P)\cap S^{31}(1)\subset C_A^I$. The proof of (\ref{CAindef}) and thus the proof of Proposition \ref{C_A F} is now complete.
\end{proof}

As the last content of this subsection, we will concentrate on the definite case. As in Section 6.6 of \cite{FKM81}, we can extend a Clifford system $\{P_0, \cdots, P_8\}$ on $\mathbb{R}^{32}$
to a system $\{P_0, \cdots, P_9\}$. Recall $P=:P_0\cdots P_8$. With this preparation, we are ready to show the following criterion for $C_A^D$:
\begin{prop}\label{C_A G}
$x\in C_A^D \Leftrightarrow \langle x, Px\rangle^2+\langle x, P_9x\rangle^2=1,$ $x\in S^{31}(1)$.
\end{prop}
\begin{proof}
Firstly, we assume the left. Following the aforementioned illustration of the definite case, if $x\in C_A^D\subset M_-^D$, there exists
$Q_0\in\Sigma(P_0,\cdots,P_8)$ such that $Q_0Px=x$. As before, we extend $Q_0$ to a symmetric Clifford system $\{Q_0,\cdots, Q_8\}$ in such a way that $P=P_0\cdots P_8=Q_0\cdots Q_8$. Noticing that $Q_1Q_0=-Q_0Q_1$, $P_9Q_0=-Q_0P_9$ and $P_9P=-PP_9$,
we can decompose $\mathbb{R}^{32}$ as the direct sum of four $8$-dimensional subspaces
$$\mathbb{R}^{32}=(E_+(Q_0)\cap E_+(P))\oplus(E_+(Q_0)\cap E_-(P))\oplus(E_-(Q_0)\cap E_+(P))\oplus(E_-(Q_0)\cap E_-(P)),$$
and express $x=x_{++}+x_{+-}+x_{-+}+x_{--}$ with corresponding indexes.
Thus $Q_0Px=x$ suggests that $x_{+-}=x_{-+}=0$, which leaves
$$x=x_{++}+x_{--}\in(E_+(Q_0)\cap E_+(P))\oplus(E_-(Q_0)\cap E_-(P)).$$
Clearly, if $x_{++}=0$ or $x_{--}=0$, then $\langle x, Px\rangle^2=1$, and $\langle x, P_9x\rangle^2=\langle Px, P_9x\rangle^2=0$.
Hence the equation $\langle x, Px\rangle^2+\langle x, P_9x\rangle^2=1$ holds. So we only need to consider the case
that $x_{++}\neq 0$ and $x_{--}\neq 0$.

An observation shows
\begin{eqnarray}\label{Q0P}
E_+(Q_0)\cap E_+(P) &=& \mathrm{Span}\{ Q_iQ_jx_{++}~:~1\leq i, j\leq 8\}\cong\mathbb{R}^8, \\
E_-(Q_0)\cap E_-(P) &=& \mathrm{Span}\{ Q_iQ_jx_{--}~:~1\leq i, j\leq 8\}\cong\mathbb{R}^8. \nonumber
\end{eqnarray}
Moreover, since $ Q_iPQ_jP=Q_iQ_j$, we still have the equivalent conditions (\ref{CA2}, \ref{QiQjx}) of $x\in C_A^D$ in the same form as in the indefinite case. In particular,
$$Q_1Q_2:~\mathrm{Span}\{ x, Q_1Q_2x,\cdots, Q_1Q_8x\}\rightarrow \mathrm{Span}\{ x, Q_1Q_2x,\cdots, Q_1Q_8x\},$$ which leads to
 $$Q_1Q_2:~\mathrm{Span}\{ Q_1Q_3x,\cdots, Q_1Q_8x\}\rightarrow \mathrm{Span}\{ Q_1Q_3x,\cdots, Q_1Q_8x\}.$$
Noticing that $Q_1Q_2$ is an anti-selfadjoint orthogonal transformation, there exists an orthonormal basis $\{v_k, k=3,\cdots,8\}$ of $\mathrm{Span}\{ Q_1Q_3x,\cdots, Q_1Q_8x\}$,
such that
\begin{equation}\label{vk'}
Q_1Q_2(v_3)=v_6,~ Q_1Q_2(v_4)=v_7, ~Q_1Q_2(v_5)=v_8.
\end{equation}
Furthermore, there exists an orthogonal transformation from $\{Q_3, \cdots, Q_8\}$ to $\{Q'_3, \cdots, Q'_8\}$ such that $v_k=Q_1Q'_kx$.
For convenience, we still denote $Q'_k$ by $Q_k$ without confusion.
Then we can derive from (\ref{vk'}) that
\begin{equation}\label{vk}
Q_1Q_2x=Q_6Q_3x=Q_7Q_4x=Q_8Q_5x,
\end{equation}
and thus
$$x\in E_+(Q_1Q_2Q_3Q_6)\cap E_+(Q_1Q_2Q_4Q_7)\cap E_+(Q_1Q_2Q_5Q_8).$$
Next we show that $Q_3Q_4Q_5x, Q_3Q_4Q_8x\in \mathrm{Span}\{Q_1x,Q_2x\}$. By (\ref{QiQjx}), it suffices to verify that
$Q_4Q_5x$ and $Q_4Q_8x$ are orthogonal to $Q_3Q_kx$ for $k=3,\ldots, 8$, which can be deduced directly by applying (\ref{vk}).
Therefore, there exists an orthogonal transformation from $\{Q_1,Q_2\}$ to $\{Q'_1, Q'_2\}$ such that $Q_3Q_4Q_5x=Q'_1x$, $Q_3Q_4Q_8x=Q'_2x$.
Again, without confusion we denote $Q'_i$ by $Q_i$.
In conclusion, we obtain
$$x\in E_+(Q_1Q_2Q_3Q_6)\cap E_+(Q_1Q_2Q_4Q_7)\cap E_+(Q_1Q_2Q_5Q_8)\cap E_+(Q_1Q_3Q_4Q_5).$$

In the following, we prove
\begin{assert}\label{assert1-1}
\begin{eqnarray*}\label{4E+}
&&\mathrm{Span}\{x_{++}, x_{--}\}\\
&=&E_+(Q_1Q_2Q_3Q_6)\cap E_+(Q_1Q_2Q_4Q_7)\cap E_+(Q_1Q_2Q_5Q_8)\cap E_+(Q_1Q_3Q_4Q_5).\nonumber
\end{eqnarray*}
\end{assert}

\begin{proof}
By (\ref{Q0P}), $Q_3Q_6x_{++}\in E_+(Q_0)\cap E_+(P)$ and $Q_1Q_2x_{--}\in E_-(Q_0)\cap E_-(P)$,
thus
$\langle Q_1Q_2Q_3Q_6x_{++}, x_{--}\rangle=\langle Q_1Q_2Q_3Q_6x_{--}, x_{++}\rangle=0$. It follows from $Q_1Q_2Q_3Q_6x=x$
that $x_{++}, x_{--}\in E_+(Q_1Q_2Q_3Q_6).$ Similarly, $x_{++}, x_{--}\in E_+(Q_1Q_2Q_4Q_7),$ $E_+(Q_1Q_2Q_5Q_8)$, $E_+(Q_1Q_3Q_4Q_5).$

We are left to prove that the dimension of space $E_+(Q_1Q_2Q_3Q_6)\cap E_+(Q_1Q_2Q_4Q_7)\cap E_+(Q_1Q_2Q_5Q_8)\cap E_+(Q_1Q_3Q_4Q_5)$ is $2$.
We first observe that $Q_1Q_2Q_3Q_6$, $Q_1Q_2Q_4Q_7$ $Q_1Q_2Q_5Q_8$ $Q_1Q_3Q_4Q_5$ are symmetric orthogonal matrices with vanishing traces. Thus their $+1$ eigenspaces have dimension $16$. Notice that $E_+(Q_1Q_2Q_3Q_6)$ is an invariant space
of the anti-commuting operators $Q_1Q_2Q_4Q_7$ and $Q_4$. Thus $E_+(Q_1Q_2Q_3Q_6)\cap E_+(Q_1Q_2Q_4Q_7)$ is of dimension $8$ and further it is an invariant space of
the anti-commuting operators $Q_1Q_2Q_5Q_8$ and $Q_5$. Thus $E_+(Q_1Q_2Q_3Q_6)\cap E_+(Q_1Q_2Q_4Q_7)\cap E_+(Q_1Q_2Q_5Q_8)$ is of dimension $4$. Since $Q_1Q_2$ commutes with
$Q_1Q_2Q_3Q_6$, $Q_1Q_2Q_4Q_7$ $Q_1Q_2Q_5Q_8$ and anti-commutes with $Q_1Q_3Q_4Q_5$, the dimension of the space $E_+(Q_1Q_2Q_3Q_6)$ $\cap E_+(Q_1Q_2Q_4Q_7)\cap E_+(Q_1Q_2Q_5Q_8)\cap E_+(Q_1Q_3Q_4Q_5)$ is $2$, as desired.
\end{proof}

Notice that $\{Q_0, \cdots, Q_8, P_9\}$ is also a Clifford system on $\mathbb{R}^{32}$.
Obviously, $P_9$ commutes with $Q_1Q_2Q_3Q_6$, $Q_1Q_2Q_4Q_7$, $Q_1Q_2Q_5Q_8$ and $Q_1Q_3Q_4Q_5$ simultaneously,
thus Assertion \ref{assert1-1} implies that
\begin{equation*}
  P_9:~\mathrm{Span}\{x_{++}, x_{--}\}\rightarrow \mathrm{Span}\{x_{++}, x_{--}\}.
\end{equation*}
Since $P_9$ anti-commutes with $P=Q_0\cdots Q_8$, we can assume that $P_9x_{--}=\lambda x_{++}$ with $\lambda\neq 0$, and thus $P_9x_{++}=\lambda^{-1} x_{--}$.
Hence we obtain $x=x_{++}+x_{--}\in \mathrm{Span}\{ Px, P_9x\}$, that is, $\langle x, Px\rangle^2+\langle x, P_9x\rangle^2=1$.

Conversely, suppose $\langle x, Px\rangle^2+\langle x, P_9x\rangle^2=1$ and $x\in S^{31}(1)$. We need to prepare
\begin{assert}\label{assert1-2}
For $x\in S^{31}(1)$, $\langle x, Px\rangle^2+\langle x, P_9x\rangle^2=1$ implies $x\in M_-^D$, i.e.,
$\sum_{i=0}^8\langle Px, P_ix\rangle^2=1$.
\end{assert}

\begin{proof}
If $Px=\pm x$, (\ref{M-FG}) implies that $x\in M_-^D$.
So we assume here that $x=x_++x_-\in E_+(P)\oplus E_-(P)$ with $x_{\pm}\neq 0$.
Since $P_9$ anti-commutes with $P$, we have $P_9x_{\pm}\in E_{\mp}(P)$ and
\begin{eqnarray*}
1=\langle x, Px\rangle^2+\langle x, P_9x\rangle^2&=&(|x_+|^2-|x_-|^2)^2+4\langle x_+, P_9x_-\rangle^2\\
&\leq& (|x_+|^2-|x_-|^2)^2+4|x_+|^2|x_-|^2=1,
\end{eqnarray*}
which shows that
\begin{equation}\label{P9}
  P_9x_-=\lambda x_+ ~\mathrm{with}~\lambda^2=\frac{|x_-|^2}{|x_+|^2}.
\end{equation}
At the mean time, we decompose $\mathbb{R}^{32}$ again as
$$\mathbb{R}^{32}=(E_+(P)\cap E_+(P_0))\oplus(E_+(P)\cap E_-(P_0))\oplus(E_-(P)\cap E_+(P_0))\oplus(E_-(P)\cap E_-(P_0)),$$
and express $x_+=x_{++}+x_{+-}$ and $x_-=x_{-+}+x_{--}$ with corresponding indexes.
Then it follows from (\ref{P9}) and $P_9P_0=-P_0P_9$ that
\begin{equation}\label{P9x-+}
  P_9x_{-+}=\lambda x_{+-},~~ P_9x_{--}=\lambda x_{++},
\end{equation}
which implies
\begin{equation}\label{|x-+|2}
  |x_{-+}|^2=\lambda^2|x_{+-}|^2,  ~~|x_{--}|^2=\lambda^2|x_{++}|^2.
\end{equation}
If $x_{+}\in E_{\pm}(P_0)$, 
then $P_0x_-=P_0P_9P_9x_-=\lambda P_0P_9x_+=-\lambda P_9P_0x_+=\mp \lambda P_9x_+=\mp x_-$, and further $P_0Px=P_0P(x_++x_-)=P_0(x_+-x_-)=\pm (x_++x_-)=\pm x$,
that is, $x\in M_-^D$. Analogously, if $x_-\in E_{\pm}(P_0)$, we can also get the same conclusion.

Hence we may assume that $x_{++}$, $x_{+-}$, $x_{-+}$ and $x_{+-}$ are all non-zero. It follows that
\begin{eqnarray*}
&& E_{+}(P)\cap E_{+}(P_0)=\mathrm{Span}\{P_1x_{+-}, \cdots, P_8x_{+-}\}, \\
&& E_{+}(P)\cap E_{-}(P_0)=\mathrm{Span}\{P_1x_{++}, \cdots, P_8x_{++}\}, \\
&& E_{-}(P)\cap E_{+}(P_0)=\mathrm{Span}\{P_1x_{--}, \cdots, P_8x_{--}\}, \\
&&  E_{-}(P)\cap E_{-}(P_0)=\mathrm{Span}\{P_1x_{-+}, \cdots, P_8x_{-+}\}.
\end{eqnarray*}
In particular, we derive that
\begin{equation*}
  |x_{++}|^2=\sum_{i=1}^8\langle x_{++}, P_i\frac{x_{+-}}{|x_{+-}|}\rangle^2, ~~ |x_{--}|^2=\sum_{i=1}^8\langle x_{--}, P_i\frac{x_{-+}}{|x_{-+}|}\rangle^2,
\end{equation*}
and thus
\begin{equation*}
  \sum_{i=1}^8\langle x_{--}, P_ix_{-+}\rangle^2=|x_{--}|^2|x_{-+}|^2=\lambda^4|x_{++}|^2|x_{+-}|^2=\lambda^4\sum_{i=1}^8\langle x_{++}, P_ix_{+-}\rangle^2.
\end{equation*}
Therefore, it follows from (\ref{|x-+|2}) that
\begin{equation*}
\langle Px, P_0x\rangle^2=(|x_{++}|^2-|x_{+-}|^2-|x_{-+}|^2+|x_{--}|^2)^2=(1+\lambda^2)^2(|x_{++}|^2-|x_{+-}|^2)^2,
\end{equation*}
and from (\ref{P9x-+}) that
\begin{eqnarray*}
\sum_{i=1}^8\langle Px, P_ix\rangle^2&=& \sum_{i=1}^8\langle x_{++}+x_{+-}-x_{-+}-x_{--}, P_i(x_{++}+x_{+-}+x_{-+}+x_{--})\rangle^2 \\
   &=& 4\sum_{i=1}^8(\langle x_{++}, P_ix_{+-}\rangle-\langle x_{--}, P_ix_{-+}\rangle)^2 \\
   &=& 4\sum_{i=1}^8(\langle x_{++}, P_ix_{+-}\rangle+\langle P_9x_{--}, P_iP_9x_{-+}\rangle)^2 \\
   &=& 4\sum_{i=1}^8(1+\lambda^2)^2\langle x_{++}, P_ix_{+-}\rangle^2\\
   &=& 4(1+\lambda^2)^2|x_{++}|^2|x_{+-}|^2.
\end{eqnarray*}
Putting the two equalities above together, we obtain
\begin{equation*}
\sum_{i=0}^8\langle Px, P_ix\rangle^2=(1+\lambda^2)^2|x_{+}|^4=|x|^4=1.
\end{equation*}
The proof of Assertion \ref{assert1-2} is complete now.
\end{proof}

Now we are in a position to show that for $x\in S^{31}(1)$, $\langle x, Px\rangle^2+\langle x, P_9x\rangle^2=1$ implies $x\in C_A^D$.
As we have shown $x\in M_-^D$,  there exists $Q_0\in \Sigma(P_0,\cdots, P_8)$ such that $Q_0Px=x$.

If $x\in E_{\pm}(P)\cap S^{31}(1)$, i.e., $Px=\delta x$ with $\delta=\pm 1$, we have $Q_0x=Px=\delta x$. Hence $x\in E_+(Q_0P)\cap E_{\delta}(Q_0)$.
On the other hand, turning (\ref{E+-}) to the definite case we get
\begin{equation}\label{E-}
E_-(Q_0P)=T_x^{\bot}M_-^D\oplus \mathrm{Span}\{Q_1Px,\cdots, Q_8Px\}.
\end{equation}
Furthermore, using a similar argument as before, one can verify that each of $E_{\pm}(Q_0P)\cap E_{\pm}(Q_0)$ has dimension $8$.
Noticing that $Q_iP$ anti-commutes with both $Q_0P$ and $Q_0$, it can be shown that
$\mathrm{Span}\{Q_1Px,\cdots, Q_8Px\}=E_-(Q_0P)\cap E_{-\delta}(Q_0)$,
and thus
$T_x^{\bot}M_-^D=E_-(Q_0P)\cap E_{\delta}(Q_0)$. This argument implies that
\begin{eqnarray*}
Q_iPQ_jP: &&\mathrm{Span}\{Q_1Px,\cdots, Q_8Px\}=E_-(Q_0P)\cap E_{-\delta}(Q_0)\\
&\rightarrow& E_-(Q_0P)\cap E_{-\delta}(Q_0)=\mathrm{Span}\{Q_1Px,\cdots, Q_8Px\},
\end{eqnarray*}
for $i,j=1,\cdots, 8, i\neq j$, which is equivalent to say $x\in C_A^D$ by (\ref{Qix}).

From now on, we may assume $x=x_++x_-\in E_+(P)\oplus E_-(P)$ with $x_{\pm}\neq 0$. Decompose $\mathbb{R}^{32}$ again as
\begin{eqnarray*}
\mathbb{R}^{32}&=&(E_+(P)\cap E_+(Q_0))\oplus(E_+(P)\cap E_-(Q_0))\\
&\oplus&(E_-(P)\cap E_+(Q_0))\oplus(E_-(P)\cap E_-(Q_0)),
\end{eqnarray*}
and express $x_+=x_{++}+x_{+-}$ and $x_-=x_{-+}+x_{--}$ with corresponding indexes.
From (\ref{P9}), we can also derive that
\begin{equation}\label{P9 definite}
  P_9x_{-+}=\lambda x_{+-},~~ P_9x_{--}=\lambda x_{++}.
\end{equation}
Since $Q_0x=Px$, we have $x_{+-}=x_{-+}=0$, i.e., $x=x_{++}+x_{--}$ and $x_{++}\neq 0$, $x_{--}\neq 0$.
A simple argument shows
\begin{eqnarray}\label{caps}
 E_{+}(P)\cap E_{+}(Q_0)&=&\mathrm{Span}\{x_{++}, Q_1Q_2x_{++}, \cdots, Q_1Q_8x_{++}\}, \nonumber\\
 E_{+}(P)\cap E_{-}(Q_0)&=&\mathrm{Span}\{Q_1x_{++}, Q_2x_{++}, \cdots, Q_8x_{++}\}, \\
 E_{-}(P)\cap E_{+}(Q_0)&=&\mathrm{Span}\{Q_1x_{--}, Q_2x_{--}, \cdots, Q_8x_{--}\}, \nonumber\\
 E_{-}(P)\cap E_{-}(Q_0)&=&\mathrm{Span}\{x_{--}, Q_1Q_2x_{--}, \cdots, Q_1Q_8x_{--}\}. \nonumber
\end{eqnarray}
Thus it follows from (\ref{P9 definite}) and (\ref{caps}) that $\langle Q_iPx, Q_jP_9x\rangle=0$, for any $i,j=1,\cdots,8$.
As it is easy to see $Q_iP_9x\in E_-(Q_0P)$, combining with (\ref{E-}), we know
\begin{equation}\label{NM_-D}
T_x^{\bot}M_-^D=\mathrm{Span}\{Q_1P_9x, \cdots, Q_8P_9x\}.
\end{equation}
Furthermore, since
\begin{eqnarray*}
Q_iQ_j=Q_iPQ_jP: && E_{+}(P)\cap E_{-}(Q_0)\rightarrow  E_{+}(P)\cap E_{-}(Q_0), \\
          && E_{-}(P)\cap E_{+}(Q_0)\rightarrow  E_{-}(P)\cap E_{+}(Q_0),
\end{eqnarray*}
we obtain that for any $1\leq i, j, k, s\leq 8$,
\begin{eqnarray*}
\langle Q_iQ_jQ_kP_9x, Q_sPx\rangle &=& \langle Q_iQ_jQ_k(\lambda^{-1}x_{--}+\lambda x_{++}), Q_s(x_{++}-x_{--})\rangle \\
   &=& -\lambda^{-1}\langle Q_iQ_jQ_kx_{--}, Q_sx_{--}\rangle +\lambda\langle Q_iQ_jQ_kx_{++}, Q_sx_{++}\rangle\\
   &=& -\lambda^{-1}\langle Q_iQ_jQ_kx_{--}, Q_sx_{--}\rangle +\lambda^{-1}\langle Q_iQ_jQ_kP_9x_{--}, Q_sP_9x_{--}\rangle \\
   &=& 0.
\end{eqnarray*}
Hence by using (\ref{NM_-D}), we finally obtain
$$Q_iQ_j:~T_x^{\bot}M_-^D\rightarrow T_x^{\bot}M_-^D,$$
which implies $x\in C_A^D$. The proof is now complete.
\end{proof}
\begin{cor}
$C_A^D$ is isometric to  $\tiny{(S^1(1)\times S^{15}(1))/\mathbb{Z}_2}$, diffeomorphic to $S^1\times S^{15}$.
\end{cor}
\begin{proof}
Noticing that $\{P, P_9\}$ constitute a Clifford system on $\mathbb{R}^{32}$,
the equation $\langle x, Px\rangle^2+\langle x, P_9x\rangle^2=1$ characterizes exactly the focal submanifold $M_-$ of the OT-FKM type isoparametric polynomial with respect to this Clifford system.
It follows from \cite{TY13} that $C_A^D\cong M_-\cong \tiny{(S^1(1)\times S^{15}(1))/\mathbb{Z}_2}$, where we identify $(t, x)\in S^1(1)\times S^{15}(1)$ with $(-t, -x)$.

To prove the second conclusion, let $\eta$ be the Hopf line bundle over $S^1$. Then $(S^1\times S^{15})/\mathbb{Z}_2$ is diffeomorphic to the
total space of the sphere bundle of $16\eta$ (\cite{TXY12}). As $2\eta$ is trivial, $(S^1\times S^{15})/\mathbb{Z}_2$ is diffeomorphic to $S^1\times S^{15}$.
\end{proof}

\section{The set $C_P$-Parallel second fundamental form}\label{sec-C_P}
This section is devoted to the determination of the set $C_P$, i.e., the set of points at which the covariant derivative of the second fundamental form
vanishes, in focal submanifolds of
an isoparametric hypersurface in a unit sphere with $g=4$.
As we mentioned before, the isoparametric families with $g=4$
are either of OT-FKM type or homogeneous with $(2, 2)$, $(4, 5)$. Firstly, we have
\begin{prop}\label{CP-homo}
For the homogeneous case with $(2, 2)$, $C_P=\emptyset$ in the focal submanifold diffeomorphic to $\mathbb{C}P^3$, and $C_P$ in the focal submanifold diffeomorphic to $\widetilde{G}_2(\mathbb{R}^5)$ is the whole focal submanifold. For the homogeneous case with $(4, 5)$, $C_P=\emptyset$  in both focal submanifolds.
\end{prop}
\begin{proof}
 Clearly, in $(2, 2)$ case, $2m_2-m_1-2=0$, $2m_1-m_2-2=0$,
the first equality in (\ref{1st}) and that in (\ref{2nd}) both hold, and therefore $C_P=C_E$ in both focal submanifolds.
However, by the homogeneity and the classification of Einstein manifolds in \cite{QTY13}, $C_E$ is completely
determined. More precisely, $C_E=\emptyset$ in the focal submanifold diffeomorphic to $\mathbb{C}P^3$, and $C_E$ in the focal submanifold diffeomorphic to $\widetilde{G}_2(\mathbb{R}^5)$ is the whole focal submanifold. As for the $(4, 5)$ case, it is easy to see that the lower bound of $\rho^{\bot}$ in (\ref{2nd}) is bigger than that
in (\ref{1st}), and thus $C_P=\emptyset$ in both submanifolds, as desired.
\end{proof}

Next, let us turn to $M_-$ of OT-FKM type. We need to prepare the following
\begin{lem}
On the focal submanifolds of isoparametric hypersurfaces with $g=4$, if at some point, the second fundamental form is parallel, so is the Ricci tensor.
\end{lem}
\begin{proof}
Recall that the Ricci tensor can be expressed by the shape operators as in (\ref{Ric }). The conclusion follows at once (cf. \cite{LZ16}).
\end{proof}

\begin{lem}
On $M_-$ of OT-FKM type with $m\geq2$, except for the cases of $(5, 2)$, $(6, 1)$ and $(9, 6)$, the Ricci tensor is not parallel at any point.
\end{lem}
\begin{proof}
In Proposition 4.2 of \cite{TY15}, except for the cases of $(5, 2)$, $(6, 1)$ and $(9, 6)$, the proof for $M_-$ of OT-FKM type with $m\geq2$ are not Ricci parallel is actually pointwise.
\end{proof}

Now we are ready to show the following
\begin{prop}\label{M-Cp}
On the focal submanifold $M_-$ of OT-FKM type, $C_P=M_-$ for $m=1$; $C_P=\emptyset$ for $m\geq 2$.
\end{prop}
\begin{proof}
Since $M_-$ with $m=1$ is homogeneous, the first conclusion follows from \cite{LZ16}.
When $m\geq 2$, notice that the three cases with $(5, 2)$, $(6, 1)$ and $(9, 6)$ are all homogeneous and
\cite{LZ16} showed that $M_-$ in these cases are not parallel.
Combining with the two lemmas above, we arrive at that $C_P=\emptyset$ for $m\geq 2$.
\end{proof}

Lastly, we will conclude this section with the following
\begin{prop}\label{M+Cp}
On the focal submanifold $M_+$ of OT-FKM type, $C_P=M_+$ in the $(2, 1)$ and $(6, 1)$ cases; $C_P=C_E=M_+$ in the $(4, 3)$ definite case, and $C_P=\emptyset$ in other cases.
\end{prop}
\begin{proof}
Comparing the first inequality in (\ref{1st}) with that in (\ref{2nd}),
we observe that if $2m_2-m_1-2=2(l-m-1)-m-2>0$, $C_P=\emptyset$ in $M_+$. So we are only left to deal with
$(m_1, m_2)=(1, 1)$, $(2, 1)$, $(4, 3)$, $(5, 2)$, and $(6, 1)$. The $(1, 1)$, $(2, 1)$, $(5, 2)$, and $(6, 1)$ cases are all
homogeneous. In the $(1, 1)$ case, $C_A=M_+$ as proved in Section 3, i.e. the upper bound of $\rho^{\bot}$ in (\ref{1st}) is achieved, thus $C_P=\emptyset$.
It was proved in \cite{LZ16} that $C_P=M_+$ in the $(2, 1)$, $(6, 1)$ cases and $C_P=\emptyset$ in the $(5, 2)$ case. While in the $(4, 3)$
case, $2m_2-m_1-2=0$, thus $C_P=C_E$. According to \cite{QTY13}, $M_+$ in the $(4, 3)$ definite case is Einstein, thus we have $C_P=C_E=M_+$.
As mentioned in Remark \ref{rem1}, the indefinite $(4, 3)$ family is congruent to the $(3, 4)$ family. Thus $M_+$ in the $(4, 3)$ indefinite case
is congruent to $M_-$ with $(3, 4)$, whose $C_P=\emptyset$ as we proved in Proposition \ref{M-Cp}.
\end{proof}

\section{The set $C_E$-Einstein condition}\label{sec-C_E}
The task of this section is to determine $C_E$.
\begin{prop}\label{CE}
On the focal submanifolds of isoparametric hypersurfaces with $g=4$, $C_E=M_+$ in $M_+$ of OT-FKM type with definite $(4, 3)$,
$C_E$ in the focal submanifold with homogeneous $(2, 2)$ diffeomorphic to $\widetilde{G}_2(\mathbb{R}^5)$ is the whole focal submanifold, and $C_E=\emptyset$ in other cases.
\end{prop}

\begin{proof}
\cite{QTY13} determined which focal submanifolds with $g=4$ are Einstein among the OT-FKM type and the homogeneous $(2, 2)$, $(4, 5)$ cases.
Then combining with the classification of isoparametric hypersurfaces in spheres, \cite{QTY13} actually classified completely the Einstein focal submanifolds with $g=4$.
More precisely, except for $M_+$ of OT-FKM type with definite $(4, 3)$, which is diffeomorphic to $Sp(2)$,
and the focal submanifold with $(2, 2)$, which is diffeomorphic to the oriented $\tilde{G}_2(\mathbb{R}^5)$, all the other
focal submanifolds are not Einstein.

As for the OT-FKM type, we notice that the proof in Theorem 1.2 (i) in \cite{QTY13} for $M_-$ to be non-Einstein is pointwise, thus $C_E=\emptyset$; and except for the five cases listed in formula $(15)$ of \cite{QTY13}, the proof for $M_+$ to be non-Einstein is also pointwise, thus $C_E=\emptyset$. So we only need to calculate
$C_E$ in those five cases. Namely, $M_+$ with multiplicities
$$(m_1, m_2)=(m, l-m-1)=(4, 3)(\mathrm{indefinite}), (7, 8), (8, 7), (9, 6), (10, 21).$$

As mentioned in Section \ref{sec-C_P}, in $M_+$ of OT-FKM type with $(4, 3)$, $C_P=C_E$, which is $\emptyset$ in the indefinite case by Proposition \ref{M+Cp}.
Besides, since $M_+$ with $(9, 6)$ and definite $(8, 7)$ are homogeneous, $C_E=\emptyset$ according to the classification in \cite{QTY13}. So we are left to verify pointwisely on $M_+$ with
$(7, 8)$, indefinite $(8, 7)$ and $(10, 21)$.

Recall the expression of Ricci curvature of $X\in T_xM_+$ (\cite{QTY13}):
$$Ric(X)=2(l-m-2)|X|^2+2\sum_{\alpha,\beta=0,\alpha<\beta}^{m}\langle X, P_{\alpha}P_{\beta}x\rangle^2.$$

\noindent
\textbf{(1) the $(7, 8)$ case}. Let us define $P=:P_0\cdots P_7$ as before, and a closed subset $D=:\{x\in M_+~|~Px=\pm x\}\subset M_+$. Given $x\in M_+$, and unit $X\in T_xM_+$, we define
a smooth function on the unit tangent bundle $S(TM_+)$:
$$f=:f(x, X)=:\sum_{\alpha,\beta=0,\alpha<\beta}^{7}\langle X, P_{\alpha}P_{\beta}x\rangle^2.$$ Furthermore, the function $f$ gives rise to a function $g$ on $M_+$ by
$$g(x)=:\min_{X\in T_xM_+,~|X|=1}f(x, X).$$
Clearly, $g$ is a continuous function, since the unit tangent bundle is locally trivial.
If $x\in M_+\backslash D$, we choose $X_0=\frac{Px-\langle Px, x\rangle x}{|Px-\langle Px, x\rangle x|}$. We need to verify that $X_0$ is a unit tangent vector at $x$. In fact, it is easy to see $\langle X_0, x\rangle=\langle X_0, P_{\alpha}x\rangle=0$, $\alpha=0,\cdots, 8$, thus $X_0\in T_xM_+$.  Furthermore, $f(x, X_0)=0$,
since $\langle Px, P_{\alpha}P_{\beta}x\rangle=\langle x, P_{\alpha}P_{\beta}x\rangle=0$ for $\alpha\neq \beta$. This means that
$g|_{M_+\backslash D}\equiv 0$. Noticing that $M_+\backslash D$ in $M_+$ is open and dense, we obtain $g(x)\equiv 0$. However, by the definition, $\max\limits_{X\in T_xM_+,~|X|=1}f(x, X)\geq 1$. Therefore, at any $x\in M_+$, $f(x, X)$ is not constant. In other words, $C_E=\emptyset$ in $M_+$.
\vspace{2mm}

\noindent
\textbf{(2) the indefinite $(8, 7)$ case.} Without loss of generality, we first give a concrete expression of Clifford system on $\mathbb{R}^{32}\cong\mathbb{O}^{4}$.
Write $x\in\mathbb{R}^{32}$ as $x=(u, v)=(u_1, u_2, v_1, v_2)\in\mathbb{O}^{4}$. Let
\begin{eqnarray*}\label{clifford system}
&& P_0x=(u_1, u_2, -v_1, -v_2),~~P_1x=(v_1, v_2, u_1, u_2),\\
&& P_{1+\alpha}x=(e_{\alpha}v_1, -e_{\alpha}v_2, -e_{\alpha}u_1, e_{\alpha}u_2), ~~~\alpha=1,\cdots,7,
\end{eqnarray*}
where $\{1,e_1,e_2,\cdots, e_7\}$ is the standard orthonormal basis of the octonions (Cayley numbers) $\mathbb{O}$ as before.
\begin{lem}\label{lem5-1}
 $P_0\cdots P_8x=(-u_1, u_2, -v_1, v_2)$.
\end{lem}

The proof is similar to that of Assertion \ref{assert3} in Section \ref{sec-C_A}, and is omitted here.
One can verify without much difficulty that $x\in M_+$ if and only if
\begin{equation}\label{M+87}
  |u|^2=|v|^2=1/2,~~\mathrm{Re}(u_1\bar{v}_1+u_2\bar{v}_2)=0,~~\mathrm{Im}(u_1\bar{v}_1-u_2\bar{v}_2)=0.
\end{equation}
The last two equalities imply that $u_1\bar{v}_1+v_2\bar{u}_2=0$. Thus $|u_1|=|v_2|$ and $|u_2|=|v_1|$, which lead to
\begin{equation}\label{87CE}
  \langle P_0\cdots P_8x, x\rangle=-|u_1|^2+|u_2|^2-|v_1|^2+|v_2|^2=0
\end{equation}
by the preceding lemma.
Define a closed subset $D'=:\{x\in M_+~|~P_1\cdots P_8x=\pm x\}\subset M_+$. Notice that $D'$ has measure zero in $M_+$ since $P_1\cdots P_8x=\pm x$ will lead to
$u_1=v_2=0$ or $u_2=v_1=0$. Analogously as in the $(7, 8)$ case, given $x\in M_+$, and unit $X\in T_xM_+$, we define
a smooth function on the unit tangent bundle $S(TM_+)$:
$$f=:f(x, X)=:\sum_{\alpha,\beta=0,\alpha<\beta}^{8}\langle X, P_{\alpha}P_{\beta}x\rangle^2.$$ Thus $f$ gives rise to functions $g$ and $G$ on $M_+$ by
$$g(x)=:\min_{X\in T_xM_+,~|X|=1}f(x, X),\quad G(x)=\max\limits_{X\in T_xM_+,~|X|=1}f(x, X).$$
Again, $g$ and $G$ are continuous, since the unit tangent bundle is locally trivial.
If $x\in M_+\backslash D'$, we choose $$X_0=\frac{P_1\cdots P_8x-\langle P_1\cdots P_8x, x\rangle x}{|P_1\cdots P_8x-\langle P_1\cdots P_8x, x\rangle x|}.$$
It is easy to see that $X_0\in T_xM_+$, since $\langle X_0, x\rangle=0$, $\langle X_0, P_0x\rangle=0$ by (\ref{87CE}) and $\langle X_0, P_{\beta}x\rangle=0$, $\beta=1,\cdots, 8$.
Furthermore, it follows from (\ref{M+87}) and Lemma \ref{lem5-1} that
$$|P_1\cdots P_8x-\langle P_1\cdots P_8x, x\rangle x|^2=1-\langle P_1\cdots P_8x, x\rangle^2=16|u_1|^2|v_1|^2,$$
and
\begin{eqnarray*}
\sum_{\alpha,\beta=0,\alpha<\beta}^{8}\langle P_1\cdots P_8x, P_{\alpha}P_{\beta}x\rangle^2 &=& \sum_{\alpha=1}^{8}\langle P_1\cdots P_8x, P_0P_{\alpha}x\rangle^2 \\
   &=& \langle P_0\cdots P_8x, P_1x\rangle^2+\sum_{\alpha=2}^{8}\langle P_0\cdots P_8x, P_{\alpha}x\rangle^2 \\
   &=& 4|\mathrm{Re}(u_1\bar{v}_1-u_2\bar{v}_2)|^2+4|\mathrm{Im}(u_1\bar{v}_1+u_2\bar{v}_2)|^2\\
      &=& 4|u_1\bar{v}_1-u_2\bar{v}_2|^2+4|u_1\bar{v}_1+u_2\bar{v}_2|^2\\
   &=&16|u_1|^2|v_1|^2,
\end{eqnarray*}
thus $f(x, X_0)=1.$ This means that $g|_{M_+\backslash D'}\leq 1$ and thus $g(x)\leq 1$ on $M_+$ by continuity.

On the other hand, it is obvious that $G(x)\geq 1$ since $P_{\alpha}P_{\beta}x$ are unit tangent vectors. In fact, we have $G(x)>1$ on $M_+$. Otherwise, suppose $f(x, X)\leq 1$ for any $X\in S(T_xM_+)$ at some $x\in M_+$.
Then for any two distinct pairs $(\alpha, \beta)$ and $(\gamma, \delta)$,
$P_{\alpha}P_{\beta}x$ and $P_{\gamma}P_{\delta}x$ must be perpendicular to each other, which is impossible since
$\sharp\{(\alpha, \beta)~|~0\leq\alpha<\beta\leq 8\}=36>\dim M_+=22.$

If $f$ is constant at some point $x\in M_+$, then $1<G(x)=g(x)\leq 1$, an obvious contradiction.
Therefore, $C_E=\emptyset$ in $M_+$.
\vspace{2mm}

\noindent
\textbf{(3) the $(10, 21)$ case}. Let $\dim M_+=52=:p$ and $\sharp\{(\alpha, \beta)~|~0\leq\alpha<\beta\leq10\}=55=:q$ for the sake of convenience.
Given $x\in M_+$, choosing an orthonormal basis of $T_xM_+$, we can identify $T_xM_+$ with the Euclidean space $\mathbb{R}^p$.
In this way, the unit vectors $P_0P_1x,P_0P_2x,\cdots,P_0P_{10}x,P_1P_2x,\cdots,P_9P_{10}x$
become unit vectors $b_1,\cdots,b_q$ in $\mathbb{R}^p$.
Write $b_i=(b_{i1}, \cdots, b_{ip})$ with $i=1,\cdots, q$, thus we obtain a matrix
$B=(b_{ij})_{q\times p}$. Define a function $$f(X)=:\sum_{i=1}^{q}\langle X, b_i\rangle^2~ for ~X\in S^{p-1}(1).$$
Clearly, $x\in C_E$ if and only if the function $f$ is a constant. On the other hand, $f$ is a quadratic form with respect to the symmetric matrix $\sum_{i=1}^qb_i^tb_i=B^tB$, thereby
$$f\equiv Const=:c ~~\Leftrightarrow~~
B^tB=c I_p.$$
As $b_i$'s are unit vectors, by taking trace we see $c=\frac{q}{p}$ and thus $f$ is a constant if and only if
$$ \langle (b_{1i},\cdots,b_{qi})^t, (b_{1j},\cdots,b_{qj})^t\rangle=\delta _{ij}\frac{q}{p}, ~for ~i,j=1,\cdots,p.$$
It is easily seen that the matrix $\sqrt{p/q}B$ can be extended to an orthogonal matrix, say $C=(c_{ij})_{q\times q}$. Since
$\langle b_{\alpha},b_{\beta}\rangle=\langle P_0P_{\alpha}x, P_0P_{\beta}x\rangle=0$ for any $\alpha\neq\beta$, $\alpha,\beta=1,\cdots, 10$,
it follows from $CC^t=C^tC=I_{q}$ that
the vectors $\beta_1,\cdots, \beta_{10}\in\mathbb{R}^{q-p}$ defined by $\beta_i=(c_{i(q-2)}, c_{i(q-1)}, c_{iq})$ $(i=1,\cdots, 10)$
satisfy $|\beta_1|^2=\cdots=|\beta_{10}|^2=1-p/q$ and $\langle \beta_i, \beta_j\rangle=0$ for any $i\neq j, i, j=1,\cdots, 10$.
However, this is impossible. Therefore, $C_E=\emptyset$ in this case.
\end{proof}

\begin{rem}
In \cite{QTY13}, they gave a sufficient condition on $M_+$ for $x\not\in C_E$. More precisely, $x$ does not belong to $C_E$ if $\dim M_+>\binom{m+1}{2}$.
Without difficulty, the arguments in the $(10, 21)$ case above can weaken the sufficient condition to  $\dim M_+>\binom{m+1}{2}-m$.
\end{rem}

\section{The focal submanifolds with $g=4$ and multiplicities $(8,7)$}\label{sec-prop1.2}
The purpose of this section is to give a proof of Proposition \ref{prop} concerning the isoparametric family of OT-FKM type with multiplicities $(8,7)$.
Firstly, we define a function $h$ on $M_-^I$ by $h(x)=\langle Px, x\rangle$,
where $P=:P_0\cdots P_8$ as defined in the introduction. We will show
\begin{lem}\label{lem6-1}
$h$ is an
isoparametric function on $M_-^I$ satisfying
\begin{equation*}\label{hM-F}
\left\{\begin{array}{ll}
|\nabla^I h|^2=4(1-h^2)\\
~~\Delta^Ih=-32h,
\end{array}\right.
\end{equation*}
where $\nabla^I h$ and $\Delta^Ih$ are the gradient and Laplacian of $h$ on $M_-^I$ with the induced metric from $S^{31}(1)$.
\end{lem}
\begin{proof}
We still make use of the notations in Subsection \ref{subsec-87}.
At first, it is easy to see that
$$\frac{1}{2}\nabla^I h=(Px)^{T}=Px-\langle Px, x\rangle x-\sum_{\alpha=1}^8\langle Px, N_{\alpha}\rangle N_{\alpha},$$
where $N_1,\cdots,N_8$ is an orthonormal basis of $T_x^{\bot}M_-^I$ in $T_xS^{31}(1)$.

Given $x \in M_-$ and $Q_0\in \Sigma(P_0,\cdots,P_8)$ with $Q_0x=x$, we can extend $Q_0$ to such a symmetric Clifford system $\{Q_0, Q_1, \cdots , Q_8\}$
with $Q_i\in\Sigma_{Q_0}$ $(i\geq 1)$ that $\Sigma(Q_0, Q_1, \cdots, Q_8)=\Sigma(P_0, P_1, \cdots, P_8)$ and $P=P_0\cdots P_8=Q_0\cdots Q_8$.
Thus we obtain
$$\langle Px, N_{\alpha}\rangle N_{\alpha}=\langle Q_1\cdots Q_8x, N_{\alpha}\rangle N_{\alpha}.$$
Clearly, the interpretation (\ref{E+-}) tells that $x\in E_+(Q_0)$ and $N_{\alpha}\in E_-(Q_0)$. Thus combining with the property that
$Q_i$ ($i\geq 1$) maps $E_{\pm}(Q_0)$ to $E_{\mp}(Q_0)$, we derive that $Q_1\cdots Q_8x\in E_+(Q_0)$, which
leads to $\langle Q_1\cdots Q_8x, N_{\alpha}\rangle=0$. Therefore, the arguments above imply
$$\frac{1}{2}\nabla^I h=Px-\langle Px, x\rangle x.$$
In particular, $$|\nabla^I h|^2=4(1-h^2).$$

Next, we turn to calculate the laplacian, i.e., the trace of Hessian $Hess_h^I$:
\begin{eqnarray}\label{hessian}
Hess_h^I(X, Y) &=& \langle \nabla^I_X\nabla^Ih, Y\rangle=2 \langle D_X(Px-hx), Y\rangle\\
&=&2 \langle PX-hX, Y\rangle,\nonumber
\end{eqnarray}
for $X, Y\in T_xM_-^I$, where $\nabla^I$ and $D$ are the Levi-Civita connections on $M_-^I$ and $\mathbb{R}^{32}$, respectively.
Notice that $\mathrm{Trace} P=\mathrm{Trace}(P_0\cdots P_8)=0$ in the indefinite case.
By virtue of the orthonormal basis (\ref{o.n.b of TxM2}) given in \cite{TY15},
for any given $i=1,\cdots,8$, we obtain
\begin{eqnarray*}
0&=&\mathrm{Trace} P\\
&=& \sum_{\alpha=1}^8\langle PN_{\alpha}, N_{\alpha}\rangle+\sum_{j=1}^8\langle PQ_jx, Q_jx\rangle+\sum_{\alpha=1}^8\langle PQ_iN_{\alpha}, Q_iN_{\alpha}\rangle+ \sum_{j=1}^8\langle PQ_iQ_jx, Q_iQ_jx\rangle\\
  &=& 2 \sum_{\alpha=1}^8\langle PN_{\alpha}, N_{\alpha}\rangle+16h,
\end{eqnarray*}
which implies $8h=-\sum_{\alpha=1}^8\langle PN_{\alpha}, N_{\alpha}\rangle$.
Substituting this into the expression of $\Delta^Ih$, we get $$\Delta^Ih=-32h.$$
Thus $h$ is an isoparametric function on $M_-^I$.
Furthermore, the focal variety of $h$ is $h^{-1}(1)\cup h^{-1}(-1)=(E_{+}(P)\cap M_-^I)\cup(E_{-}(P)\cap M_-^I)=S^{15}(1)\sqcup S^{15}(1)=C_A^I$ by Proposition \ref{C_A F}.
\end{proof}

For the definite case, as in the introduction, we will use the symmetric Clifford system $\{\tilde{P}_{0}=:P_{0}P, \cdots, \tilde{P}_{8}=:P_{8}P\}$.  Define the function $\tilde{h}$
on $M_-^D$ by $\tilde{h}(x)=\langle Px, x\rangle$ with $P=:P_0\cdots P_8$.
It should be reminded that $h$ and $\tilde{h}$ have the same expression but different domains.
Then the same arguments as above lead us to
\begin{lem}\label{lem6-2}
$\tilde{h}$ is an
isoparametric function on $M_-^D$ satisfying
\begin{equation*}\label{hM-F}
\left\{\begin{array}{ll}
|\nabla^D \tilde{h}|^2=4(1-\tilde{h}^2)\\
~~\Delta^D\tilde{h}=-32\tilde{h},
\end{array}\right.
\end{equation*}
where $\nabla^D \tilde{h}$ and $\Delta^D\tilde{h}$ are the gradient and Laplacian of $\tilde{h}$ on $M_-^D$ with the induced metric from $S^{31}(1)$.
The focal variety of $\tilde{h}$ is $\tilde{h}^{-1}(1)\cup \tilde{h}^{-1}(-1)=(E_{+}(P)\cap M_-^D)\cup(E_{-}(P)\cap M_-^D)=S^{15}(1)\sqcup S^{15}(1)$ by Proposition \ref{C_A F}.
\hfill$\Box$
\end{lem}

\begin{rem}
According to \cite{Wan87} or \cite{GT13}, the existence of the isoparametric function $h$ on $M_-^I$ with focal variety $S^{15}(1)\sqcup S^{15}(1)$ gives rise to a decomposition
of the manifold $M_-^I$. More precisely, $M_-^I$ is a union of two $D^8$-bundles over $S^{15}$. There exists the same decomposition on $M_-^D$. However, by Corollary 1 in \cite{Wan88}, $M_-^I$ is diffeomorphic to $S^8\times S^{15}$. By \cite{FKM81} or \cite{Wan88}, $M_-^D$ is an $S^8$ bundle over $S^{15}$. But according to Theorem 1 in \cite{Wan88}, $M_-^I$ and $M_-^D$ are not homeomorphic.
\end{rem}

As the last part of this paper, we want to show that both of $M_+^I\hookrightarrow M_-^D$ and $M_+^D\hookrightarrow M_-^I$ are totally isoparametric and austere hypersurfaces.
For the convenience of notations, we will only prove the case $M_+^I\hookrightarrow M_-^D$, as the proof of the other will use the same
method, just replacing $P_{\alpha}$ and $\tilde{P}_{\alpha}=:P_{\alpha}P$ with $\tilde{P}_{\alpha}$ and $P_{\alpha}$, respectively.

The first step is to establish
\begin{lem}\label{lem6-3}
 $M_+^I=\tilde{h}^{-1}(0)\hookrightarrow M_-^D$ is an isoparametric hypersurface.
\end{lem}

\begin{proof}
We only need to show $M_+^I\subset M_-^D$, since then for $x\in M_+^I$, $\tilde{h}(x)=\langle Px, x\rangle=\langle Px, Q_0Px\rangle=\langle x, Q_0x\rangle=0$, i.e., $M_+^I\subset \tilde{h}^{-1}(0)$;
 and conversely for $x\in\tilde{h}^{-1}(0)\subset M_-^D$, $\langle Q_ix,x\rangle=\langle Q_ix, Q_0Px\rangle=0$ for $i=1,\ldots,8$ and $\langle Q_0x,x\rangle=\langle Q_0x, Q_0Px\rangle=\tilde{h}(x)=0$, thus $\tilde{h}^{-1}(0)\subset M_+^I$.

By a similar argument as in Proposition \ref{C_A G}, we decompose $\mathbb{R}^{32}$ as the direct sum of four $8$-dimensional subspaces :
\begin{eqnarray*}
&&\mathbb{R}^{32}= E_+(P_0)\oplus E_-(P_0) \\
  &=& (E_+(P_0)\cap E_+(P))\oplus(E_+(P_0)\cap E_-(P))\oplus(E_-(P_0)\cap E_+(P))\oplus(E_-(P_0)\cap E_-(P)),
\end{eqnarray*}
and write $x\in S^{31}(1)$ as $x=x_++x_-=x_{++}+x_{+-}+x_{-+}+x_{--}$ with corresponding indexes.
Given $x\in M_+^I\subset S^{31}(1)$, the property $\langle P_0x, x\rangle=0$ indicates that $|x_+|^2=|x_-|^2=\frac{1}{2}$,
and $\langle P_{\alpha}x, x\rangle=0$ ($\alpha\geq 1$) indicates that $\langle P_{\alpha}x_+, x_-\rangle=0$ and further
\begin{equation}\label{subset}
\langle P_{\alpha}x_{++}, x_{-+}\rangle+\langle P_{\alpha}x_{+-}, x_{--}\rangle=0,
\end{equation}
since $P_{\alpha}P_0+P_0P_{\alpha}=0$ and $P_{\alpha}P=PP_{\alpha}$ for $\alpha\geq 1$.

Clearly, if $x_{++}=0$, then $x_{+-}=x_+\neq 0$, and
$\mathrm{Span}\{P_1x_{+-},\cdots, P_8x_{+-}\}=E_-(P_0)\cap E_-(P)$ for dimensional reason.
It follows from (\ref{subset}) that $x_{--}=0$. Hence
$$Px=P(x_{+-}+x_{-+})=-x_{+-}+x_{-+}=-P_0(x_{+-}+x_{-+})=-P_0x,$$
which indicates that $-P_0Px=x$, and thus  $x\in M_-^D$.
Analogously, if any one of $x_{+-}$, $x_{-+}$ and $x_{--}$ vanishes, we can also show $x\in M_-^D$ by the same way.

Now suppose none of $x_{++}$, $x_{+-}$, $x_{-+}$ and $x_{--}$ is zero. Then we have
\begin{eqnarray*}
&& \mathrm{Span}\{P_1x_{-+},\cdots, P_8x_{-+}\}=E_+(P_0)\cap E_+(P),\\
&& \mathrm{Span}\{P_1x_{--},\cdots, P_8x_{--}\}=E_+(P_0)\cap E_-(P),\\
&& \mathrm{Span}\{P_1x_{++},\cdots, P_8x_{++}\}=E_-(P_0)\cap E_+(P),\\
&& \mathrm{Span}\{P_1x_{+-},\cdots, P_8x_{+-}\}=E_-(P_0)\cap E_-(P),
\end{eqnarray*}
and thus
\begin{eqnarray*}
x_{++}=\sum_{\alpha=1}^8\langle x_{++}, P_{\alpha}\frac{x_{-+}}{|x_{-+}|}\rangle P_{\alpha}\frac{x_{-+}}{|x_{-+}|},&&
 x_{+-}=\sum_{\alpha=1}^8\langle x_{+-}, P_{\alpha}\frac{x_{--}}{|x_{--}|}\rangle P_{\alpha}\frac{x_{--}}{|x_{--}|},\\
x_{-+}=\sum_{\alpha=1}^8\langle x_{-+}, P_{\alpha}\frac{x_{++}}{|x_{++}|}\rangle P_{\alpha}\frac{x_{++}}{|x_{++}|},&&
x_{--}=\sum_{\alpha=1}^8\langle x_{--}, P_{\alpha}\frac{x_{+-}}{|x_{+-}|}\rangle P_{\alpha}\frac{x_{+-}}{|x_{+-}|},
\end{eqnarray*}
which implies
\begin{equation*}
  |x_{++}|^2|x_{-+}|^2=\sum_{\alpha=1}^8\langle x_{++}, P_{\alpha}x_{-+}\rangle^2,\quad |x_{+-}|^2|x_{--}|^2=\sum_{\alpha=1}^8\langle x_{+-}, P_{\alpha}x_{--}\rangle^2.
\end{equation*}
Substituting these into (\ref{subset}), we obtain
\begin{equation*}
|x_{++}|^2=|x_{--}|^2=:a,\quad |x_{+-}|^2=|x_{-+}|^2=:b,~\mathrm{with}~a+b=\frac{1}{2},
\end{equation*}
which implies directly that $\tilde{h}(x)=\langle Px, x\rangle=|x_{++}|^2-|x_{+-}|^2+|x_{-+}|^2-|x_{--}|^2=0$.
Besides, a further calculation shows $\langle P_0x, Px\rangle=2(a-b)$ and
$$\sum_{\alpha=1}^8\langle P_{\alpha}x, Px\rangle^2=16\sum_{\alpha=1}^8\langle x_{++}, P_{\alpha}x_{-+}\rangle^2=16ab,$$
which lead us to $\sum_{\alpha=0}^8\langle P_{\alpha}x, Px\rangle^2=4(a+b)^2=1$, i.e., $x\in M_-^D$.


In conclusion, $M_+^I=\tilde{h}^{-1}(0)\hookrightarrow M_-^D$ is an isoparametric hypersurface in $M_-^D$ by Lemma \ref{lem6-2}.
\end{proof}

\begin{rem}
Using the same method, we can also prove that $M_+^D=h^{-1}(0)\hookrightarrow M_-^I$ is an isoparametric hypersurface in $M_-^I$.
\end{rem}

Our final task is to prove
\begin{lem}\label{lem6-4}
$M_+^I\hookrightarrow M_-^D$ is a totally isoparametric, austere hypersurface in $M_-^D$.
\end{lem}

\begin{proof}
We have proved that $M_+^I\hookrightarrow M_-^D$ is an isoparametric hypersurface.
According to \cite{GTY15}, to prove that $M_+^I\hookrightarrow M_-^D$ is totally isoparametric, it suffices to prove that
the eigenvalues of $Hess_{\tilde{h}}$ are constant on any regular level set of $\tilde{h}$.

Given $x\in M_-^D$ with $Q_0Px=x$, as in (\ref{E+-}), we would like to represent $E_{\pm}(Q_0P)$ as
\begin{eqnarray*}
&&E_-(Q_0P) = \mathrm{Span}\{N_1,\cdots, N_8,~Q_1Px,\cdots,Q_8Px\} \\
&&E_+(Q_0P) = \mathrm{Span}\{Q_1PN_1,\cdots, Q_1PN_8,~x, Q_1Q_2x,\cdots,Q_1Q_8x\}.
\end{eqnarray*}
In fact, as we mentioned in Subsection \ref{subsec-87}, $\{N_1,\cdots, N_8\}$ is an orthonormal basis of $T_x^{\bot}M_-^D$, and
$\{Q_1Px,\cdots,Q_8Px, Q_1PN_1,$ $\cdots, Q_1PN_8, Q_1Q_2x,\cdots,Q_1Q_8x\}$ is an orthonormal basis of $T_xM_-^D$.

Under an appropriate arrangement of the orthonormal basis of $\mathbb{R}^{32}=T_xM_-^D\oplus T_x^{\bot}M_-^D\oplus \mathrm{Span}\{x\}$, we can split the matrix $P$ into block form (up to an adjoint transformation by an orthogonal matrix) as
\begin{eqnarray*}
P &=& \left(
  \begin{array}{cccc}
    \langle PQ_iPx, Q_jPx\rangle & \langle PQ_iPx, Q_1PN_{\alpha}\rangle &  \langle PQ_iPx, Q_1Q_jx\rangle & \langle PQ_iPx, N_{\alpha}\rangle\\
    \langle PQ_1PN_{\beta}, Q_jPx\rangle & \langle PQ_1PN_{\beta}, Q_1PN_{\alpha}\rangle  &  \langle PQ_1PN_{\beta},Q_1Q_jx\rangle &\langle PQ_1PN_{\beta}, N_{\alpha}\rangle  \\
    \langle PQ_1Q_ix, Q_jPx\rangle  &  \langle PQ_1Q_ix, Q_1PN_{\alpha}\rangle &  \langle PQ_1Q_ix, Q_1Q_jx\rangle  &  \langle PQ_1Q_ix, N_{\alpha}\rangle\\
    \langle PN_{\beta}, Q_jPx\rangle  &  \langle PN_{\beta}, Q_1PN_{\alpha}\rangle  &   \langle PN_{\beta}, Q_1Q_jx\rangle  &  \langle PN_{\beta}, N_{\alpha}\rangle
   \end{array}
\right) \\
  &=&  \left(
  \begin{array}{cccc}
  \tilde{h}I_8  &  0   &   0   &   U \\
  0     &  \mathcal{P}_{\bot}  &  U^t   &  0\\
  0     &  U   &   \tilde{h}I_8   &  0  \\
  U^t   &  0   &  0   &  \mathcal{P}_{\bot}
  \end{array}\right),
\end{eqnarray*}
where $i,j,\alpha,\beta=1,\cdots, 8$ and $U=:(\langle Q_ix, N_{\alpha}\rangle)_{8\times8}$, $\mathcal{P}_{\bot}=:(\langle PN_{\alpha}, N_{\beta}\rangle)_{8\times8}$.
Here the second equality follows from the equality $Q_0Px=x$ and the property that $Q_i$ maps $E_{\pm}(Q_0P)$ to $E_{\mp}(Q_0P)$ $(i\geq 1)$.
On the other hand, from
\begin{eqnarray*}
I_{32}&=& P^2 \\
    &=&  \left(
  \begin{array}{cccc}
  \tilde{h}^2I_8+UU^t  &  0   &   0   &   \tilde{h}U+U\mathcal{P}_{\bot} \\
  0     &  (\mathcal{P}_{\bot})^2+U^tU  &  (\mathcal{P}_{\bot})U^t+\tilde{h}U^t   &  0\\
  0     &  U(\mathcal{P}_{\bot})+\tilde{h}U   &   UU^t+\tilde{h}^2I_8   &  0  \\
 \tilde{ h}U^t+(\mathcal{P}_{\bot})U^t   &  0   &  0   &  U^tU+(\mathcal{P}_{\bot})^2
  \end{array}\right),
\end{eqnarray*}
we derive that $UU^t=(1-\tilde{h}^2)I_8$, $(\mathcal{P}_{\bot})^2+U^tU=I_8$ and $(\mathcal{P}_{\bot}+\tilde{h}I_8)U^t=0$.
Since we need only to calculate $Hess_{\tilde{h}}$ on the regular level sets, i.e., on the points with $\tilde{h}^2<1$, we
are guaranteed that $\mathrm{rank} U=8$ and $\mathcal{P}_{\bot}=-\tilde{h}I_8$, which leaves us to
\begin{equation*}
  P=\left(
  \begin{array}{cccc}
  \tilde{h}I_8  &  0   &   0   &   U \\
  0     &  -\tilde{h}I_8  &  U^t   &  0\\
  0     &  U   &   \tilde{h}I_8   &  0  \\
  U^t   &  0   &  0   &  -\tilde{h}I_8
  \end{array}\right).
\end{equation*}
We decompose $U_{8\times 8}$ as
\begin{equation*}
  U_{8\times 8}=\binom{u}{(U_1)_{7\times 8}},
\end{equation*}
with $u=(\langle Px, Q_1PN_1\rangle,\cdots, \langle Px, Q_1PN_8\rangle)$.
Restricting on $T_xM_-^D$ with $\tilde{h}^2(x)< 1$, the matrix of the quadratic form associated with $P$
can be expressed as
\begin{equation*}
   \mathcal{P}_{\top}= \left(\begin{array}{ccc}
 \tilde{ h}I_8  &  0   &   0   \\
  0     &  -\tilde{h}I_8  &  U_1^t \\
  0     &  U_1   &   \tilde{h}I_7
  \end{array}\right).
\end{equation*}

Now we are going to calculate the eigenvalues of $\mathcal{P}_{\top}$ by computing
\begin{equation*}
  \mathrm{det}(\lambda I_{23}-\mathcal{P}_{\top})=
  \begin{vmatrix}
  (\lambda-\tilde{h})I_8 & 0 & 0 \\
  0   &  (\lambda+\tilde{h})I_8  &  -U_1^t\\
  0   &  -U_1  &  (\lambda-\tilde{h})I_7
  \end{vmatrix}.
\end{equation*}

Obviously, $\lambda=\tilde{h}$ is an eigenvalue of $\mathcal{P}_{\top}$ with multiplicity at least $8$.
When $\lambda\neq \tilde{h}$, we see
\begin{equation*}
\begin{vmatrix}
  (\lambda+\tilde{h})I_8  &  -U_1^t\\
   -U_1  &  (\lambda-\tilde{h})I_7
  \end{vmatrix}=
  \begin{vmatrix}
 (\lambda+\tilde{h})I_8-(\lambda-\tilde{h})^{-1}U_1^tU_1  &  0\\
   -U_1  &  (\lambda-\tilde{h})I_7
  \end{vmatrix}=
  (\lambda+\tilde{h})(\lambda^2-1)^7.
\end{equation*}
Thus, $\mathrm{det}(\lambda I_{23}-\mathcal{P}_{\top})=(\lambda-\tilde{h})^8(\lambda+\tilde{h})(\lambda+1)^7(\lambda-1)^7$ for any $\lambda$.
Namely, the eigenvalues of $\mathcal{P}_{\top}$ on any regular level set of $\tilde{h}$ are constant. Equivalently, for any $k$, $\mathrm{Trace}(\mathcal{P}_{\top})^k$
is constant on the regular level sets of $\tilde{h}$. Recalling (\ref{hessian}), we have proved that the eigenvalues of $Hess_{\tilde{h}}$ are constant.

Therefore, $M_+^I\hookrightarrow M_-^D$ is a totally isoparametric hypersurface.
At the mean time, we notice that $\frac{\nabla^D\tilde{h}}{|\nabla^D\tilde{h}|}$ is a unit normal of $M_+^I$ in $M_-^D$, while $M_+^I$ and $M_-^D$
are both submanifolds in $S^{31}(1)$. Therefore, the eigenvalues of the shape operator associated with
$\frac{\nabla^D\tilde{h}}{|\nabla^D\tilde{h}|}$ on $M_+^I$ are the same as those in $S^{31}(1)$, i.e., $1, -1$ with multiplicity $7$ and $0$ with multiplicity $8$,
which makes $M_+^I$ an austere hypersurface in $M_-^D$.
\end{proof}


\begin{ack}
The authors would like to thank Professors R. Miyaoka and Q. S. Chi for their interest and helpful comments.

\end{ack}

\end{document}